\documentclass[12pt]{article}
\usepackage[utf8]{inputenc}
\usepackage{amsmath}
\usepackage{amssymb}
\usepackage[margin = 1 in]{geometry}
\usepackage{graphicx}
\usepackage{amsthm}
\usepackage{mathtools}
\mathtoolsset{showonlyrefs}
\usepackage{hyperref}
\hypersetup{colorlinks,linkcolor={blue}}
\usepackage{theoremref}
\usepackage{mleftright}
\usepackage{extpfeil}
\usepackage{enumitem}
\usepackage{cellspace}
\newcommand{\C}{\mathbb{C}}

\newcommand{\R}{\mathbb{R}}

\renewcommand{\H}{\mathbb{H}}
\newcommand{\im}{\operatorname{im}}
\newcommand{\id}{\text{id}}
\newcommand{\sgn}{\operatorname{sgn}}

\newcommand{\lspan}{\operatorname{span}}

\theoremstyle{definition}
\newtheorem{definition}{Definition}[section]

\theoremstyle{definition}
\newtheorem{fact}[definition]{Fact}

\theoremstyle{definition}
\newtheorem{example}[definition]{Example}

\theoremstyle{plain}
\newtheorem{theorem}[definition]{Theorem}

\theoremstyle{plain}
\newtheorem{proposition}[definition]{Proposition}

\theoremstyle{plain}
\newtheorem{lemma}[definition]{Lemma}

\theoremstyle{plain}
\newtheorem{corollary}[definition]{Corollary}

\theoremstyle{plain}

\theoremstyle{remark}
\newtheorem{remark}[definition]{Remark}

\theoremstyle{definition}
\newtheorem*{definition*}{Definition}

\theoremstyle{plain}
\newtheorem*{theorem*}{Theorem}

\theoremstyle{remark}
\newtheorem*{remark*}{Remark}

\newtheoremstyle{definitioncolon}%
  {}{}{}{}{\bfseries}{:}{ }{}
\theoremstyle{definitioncolon}

\title{On extending results of Gluck and Warner on fibrations of spheres by great subspheres}
\author{Eric Yu}
\date{\ifcase\month
  \or January\or February\or March\or April\or May\or June%
  \or July\or August\or September\or October\or November\or December\fi
  \space \the\year}

\begin{document}

\large\newlength{\oldparskip}\setlength\oldparskip{\parskip}\parskip=.3in
\thispagestyle{empty}
\vspace*{0 cm}
\begin{center}
ON EXTENDING RESULTS OF GLUCK AND WARNER ON FIBRATIONS OF SPHERES BY GREAT SUBSPHERES

Eric Yu

A THESIS

in

Mathematics

Presented to the Faculties of the University of
Pennsylvania 

in 

Partial
Fulfillment of the Requirements for the 

Degree of Master of
Arts

2026
\end{center}

\noindent Supervisor of Thesis \\ Dennis DeTurck \\ Robert A. Fox Leadership Professor of Mathematics

\noindent Graduate Group
Chair \\ Angela Gibney \\ Presidential Professor of Mathematics

\normalsize\parskip=\oldparskip

\newpage
\pagenumbering{roman}
\setcounter{page}{2}

\section*{Acknowledgments}

Special thanks to Dr. Dennis DeTurck for his excellent mentorship and support. This thesis would not have been possible without his guidance. Thanks also to Dr. Herman Gluck for serving on my defense committee, as well as my family for supporting me in my pursuit of mathematics. 

\newpage

\begin{center}
\textbf{ABSTRACT}

\bigskip

ON EXTENDING RESULTS OF GLUCK AND WARNER ON \\ FIBRATIONS OF SPHERES BY GREAT SUBSPHERES

\bigskip

Eric Yu \\ Dennis DeTurck, Advisor
\end{center}

    In this paper, we build upon the work of Gluck and Warner who showed in 1983 that the set of positively oriented fibrations of a 3-sphere by oriented great circles is in bijection with the set of distance-decreasing maps from the 2-sphere to itself. One approach to generalizing their result to higher-dimensional spheres involves understanding when exactly two Hopf fibrations of $S^{2n-1}$ are guaranteed to agree on a fiber. We give a complete characterization of this phenomenon, and we discuss the barriers which prevent us from obtaining a fully general version of Gluck and Warner's result.

\newpage
\pagenumbering{arabic}
\setcounter{page}{1}

\section{Introduction}

In 1983, Gluck and Warner discovered a nice characterization for the set of fibrations of the 3-sphere by oriented great circles:
\begin{theorem} \thlabel{gluck_warner}
    \textnormal{\cite[Theorem A]{gluck_warner}}
    The space of fibrations of $S^3$ by oriented great circles is homeomorphic to the disjoint union of two copies of $X$, where $X$ is the space of distance-decreasing maps $S^2 \to S^2$.
\end{theorem}
For the problem of classifying fibrations of spheres by great subspheres, this theorem covers the simplest nontrivial case: Since any two great circles of $S^2$ must intersect, 3 is the smallest value of $n>1$ for which $S^n$ could possibly admit a fibration by great circles. And to see that such a fibration exists, we look no further than the one famously discovered by Hopf in 1931 \cite{hopf}:
\begin{definition}[Standard Hopf fibration]\thlabel{standard_hopf}
    View $S^{2n-1}$ as the unit sphere in $\R^{2n} \cong \C^{n}$, and view $S^1$ as the complex unit circle oriented counterclockwise. $S^1$ acts on $S^{2n-1}$ via multiplication. We define the standard Hopf fibration to be the fibration of $S^{2n-1}$ whose fibers are the oriented great circle orbits of this action.
\end{definition}
\begin{definition}[Hopf fibration]\thlabel{transformation_sign}
    Let $H$ be the standard Hopf fibration of $S^{2n-1}$ in $\R^{2n}$. We define a Hopf fibration to be a fibration of the form $T(H)$, where $T: \R^{2n} \to \R^{2n}$ is an orthogonal linear transformation. If $\det(T) > 0$ then we say the Hopf fibration has positive sign. If $\det(T) < 0$ then we say it has negative sign. 
\end{definition}

The purpose of this paper is to document an attempt to generalize Gluck and Warner's result to higher-dimensional spheres. In Section~\ref{gw_argument}, we summarize the relevant parts of Gluck and Warner's argument, and in Section~\ref{motivation} we explain how we might hope to generalize it. In Section~\ref{hopf_linear}, we state the correspondence between Hopf fibrations and orthogonal complex structures, which we use in Section~\ref{hopf_agreement} to establish our main result:
\begin{theorem}\thlabel{main_result}
    Let $H_1$ and $H_2$ be two Hopf fibrations on $S^{2n-1}$, and suppose one of the following holds:
    \begin{enumerate}
        \item $H_1$ and $H_2$ have opposite sign and $n$ is even.
        \item $H_1$ and $H_2$ have the same sign and $n$ is odd.
    \end{enumerate}
    Then the fibrations $H_1$ and $H_2$ share an oriented fiber.
\end{theorem}
We prove this along with \thref{main_result_corollary}, which is slightly more general. In Section~\ref{next_steps}, we discuss possible next steps. Finally, in Section~\ref{S3_fibrations} we consider the case of fibrations by great subspheres of dimension higher than 1 and give examples which illustrate why this is less well-behaved than the case of fibrations by great circles. 

\section{Gluck and Warner's argument}\label{gw_argument}

The purpose of this section is to summarize the essential parts of Gluck and Warner's proof of \thref{gluck_warner}. 

Viewing $S^3$ as the unit sphere centered at the origin in $\R^4$, we notice that every oriented great circle of $S^3$ determines a unique oriented 2-plane in $\R^4$. This gives us a bijection between the set of oriented great circles in $S^3$ and the Grassmannian of oriented $2$-planes in $\R^4$, which we denote $G_2^+(\R^4)$.

\subsection{The homeomorphism $G_2^+(\R^4) \to S^2 \times S^2$}

It's well-known that $G_2^+(\R^4)$ is homeomorphic to $S^2 \times S^2$. Here we reproduce the explicit construction of this homeomorphism used in \cite{gluck_warner}. We write $\Lambda^k\R^4$ to denote the set of alternating multilinear maps $(\R^4)^k \to \R$, and we recall that the wedge product $\wedge: \Lambda^1\R^4 \times \Lambda^1 \R^4 \to \Lambda^2 \R^4$ is given by 
\begin{equation}
    [\alpha \wedge \beta](a, b) = \det\begin{bmatrix}
    \alpha(a) & \alpha(b) \\
    \beta(a) & \beta(b)
\end{bmatrix}.
\end{equation}

\begin{definition}[The map $\omega: G_2^+(\R^4) \to \Lambda^2 \R^4$]\thlabel{omega}
    Given an oriented 2-plane $P \subseteq \R^4$ passing through the origin, we define $\omega_P \in \Lambda^2 \R^4$ to equal $u^* \wedge v^*$, where $(u, v)$ is an orthonormal basis for $P$ with positive orientation and $u^* \in \Lambda^1\R^4$ is the map $[w \mapsto \langle w, u\rangle]$. In other words, $\omega_P$ is the function which maps a pair of vectors $a, b \in \R^4$ to the signed area of the parallelogram spanned by the orthogonal projections of $a$ and $b$ onto $P$. With this description, it is clear that $\omega$ is well-defined and injective.
    \begin{center}
        \includegraphics[width = 5.4 cm]{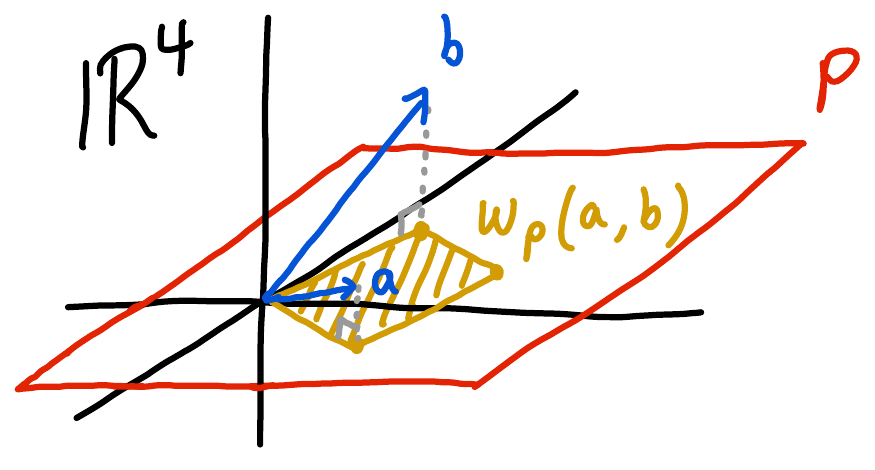}
    \end{center}
\end{definition}
In \thref{image_of_omega} we will see that $\im(\omega) \cong S^2 \times S^2$, and so $\omega$ is our desired homeomorphism.

\begin{definition}[Inner product on $\Lambda^2\R^4$]\thlabel{inner_product}
    We define the inner product on $\Lambda^2\R^4$ to be the unique inner product such that for all planes $P, Q \in G_2^+(\R^4)$ we have $\langle \omega_P, \omega_Q \rangle$ equals the determinant of the orthogonal projection map $P \to Q$. For a proof of existence and uniqueness, see Section~\ref{induced_inner_product}.
\end{definition}
\begin{proposition}\thlabel{norm_one}
    Let $P \in G_2^+(\R^4)$. Then $\langle\omega_P, \omega_P \rangle = 1$.
\end{proposition}
\begin{proof}
    The orthogonal projection map $P \to P$ is the identity, which has determinant 1.
\end{proof}

\begin{definition}[Hodge star]\thlabel{hodge_star}
    We recall that the Hodge star operator $*: \Lambda^2 \R^4 \to \Lambda^2 \R^4$ is defined to be the unique linear operator satisfying $*\omega_P = \omega_{P^\perp}$, where $P^\perp$ is the orthogonal complement of $P$ oriented so that if $(p_1,p_2)$ is a positive basis for $P$ and $(p_3, p_4)$ is a positive basis for $P^\perp$, then $(p_1,p_2, p_3, p_4)$ is a positive basis for $\R^4$. For a proof of existence and uniqueness, see Section~\ref{hodge_star_existence_uniqueness}.
\end{definition}

\begin{lemma}\thlabel{nontrivial_intersection}
\textnormal{\cite[Proposition 4.5]{gluck_warner}}
    Two planes $P$ and $Q$ in $G_2^+(\R^4)$ intersect nontrivially if and only if $\langle \omega_P, \omega_{Q^\perp}\rangle = 0$.
\end{lemma}
\begin{proof}
    Forward direction: Assuming $P$ and $Q$ intersect nontrivially, let $x$ be a nonzero vector in $P \cap Q$. Since $x$ is orthogonal to $Q^\perp$, the orthogonal projection map $P \to Q^\perp$ maps $x$ to $0$ and therefore has determinant 0. By \thref{inner_product}, we conclude that $\langle \omega_P, \omega_{Q^\perp}\rangle = 0$.
    
    Backward direction: Assuming $\langle \omega_P, \omega_{Q^\perp} \rangle = 0$, by \thref{inner_product} we have that the orthogonal projection map $P \to Q^\perp$ has determinant 0. Therefore, there exists a nonzero vector $x \in P$ whose orthogonal projection onto $Q^\perp$ is 0. In other words, $x \in Q$. So $x$ is a nonzero vector in $P \cap Q$, meaning $P$ and $Q$ intersect nontrivially.
\end{proof}

\begin{definition}[The subspaces $E_+, E_- \leq \Lambda^2\R^4$]\thlabel{eigenspaces}
    Let $\alpha$ be a vector in $\Lambda^2 \R^4$. We say that $\alpha$ is self-dual if $*\alpha = \alpha$, and we say that $\alpha$ is anti-self-dual if $*\alpha = -\alpha$. We define $E_+$ and $E_-$ to be the linear subspaces of $\Lambda^2\R^4$ consisting of all self-dual and all anti-self-dual vectors, respectively. We have an orthogonal direct sum decomposition $\Lambda^2\R^4 = E_- \oplus E_+$. To see this, we observe that 
    \begin{equation}
        \biggl\{ \frac{e_1^* \wedge e_2^* + e_3^* \wedge e_4^*}{\sqrt{2}}, \quad \frac{e_1^* \wedge e_3^* - e_2^* \wedge e_4^*}{\sqrt{2}}, \quad \frac{e_1^* \wedge e_4^* + e_2^* \wedge e_3^*}{\sqrt{2}} \biggr\}
    \end{equation}
    is an orthonormal basis for $E_+$, and
    \begin{equation}
        \biggl\{ \frac{e_1^* \wedge e_2^* - e_3^* \wedge e_4^*}{\sqrt{2}}, \quad \frac{e_1^* \wedge e_3^* + e_2^* \wedge e_4^*}{\sqrt{2}}, \quad \frac{e_1^* \wedge e_4^* - e_2^* \wedge e_3^*}{\sqrt{2}} \biggr\}
    \end{equation}
    is an orthonormal basis for $E_-$, and the union of these bases is an orthonormal basis for $\Lambda^2\R^4$. (see Proposition~\ref{basis_of_exterior}).
\end{definition}

\begin{definition}[The maps $\pi_-: \Lambda^2\R^4 \to E_-$ and $\pi_+: \Lambda^2\R^4 \to E_+$]
    Any $\alpha \in \Lambda^2\R^4$ can be split into its self-dual and anti-self-dual components:
    \begin{equation}
        \alpha = \frac{\alpha + *\alpha}{2} + \frac{\alpha - *\alpha}{2}.
    \end{equation}
    We define $\pi_+(\alpha) = \frac{1}{2}(\alpha + *\alpha)$ to be the self-dual part of $\alpha$ and $\pi_-(\alpha) = \frac{1}{2}(\alpha - *\alpha)$ to be the anti-self-dual part of $\alpha$.
\end{definition}

\begin{theorem}\thlabel{image_of_omega}
    Let $S^2_+$ and $S^2_-$ denote the spheres of radius $\frac{1}{\sqrt{2}}$ centered at the origin in $E_+$ and $E_-$, respectively. Then the image of the map $\omega: G_2^+(\R^4) \to \Lambda^2\R^4$ is $S^2_- \times S^2_+$.
\end{theorem}
\begin{proof}
    To prove the forward inclusion $\im(\omega) \subseteq S^2_- \times S^2_+$, it suffices to show that $\|\pi_-(\omega_P)\| = \|\pi_+(\omega_P)\| = \frac{1}{\sqrt{2}}$ for all $P \in G_2^+(\R^4)$.
    \begin{align}
        \|\pi_+(\omega_P)\| &= \biggl\| \frac{\omega_P + *\omega_P}{2} \biggr\| \\
        &= \frac{1}{2} \sqrt{\langle \omega_P + *\omega_P, \; \omega_P + *\omega_P \rangle} \\
        &= \frac{1}{2}\sqrt{\langle\omega_P, \omega_P\rangle + \langle *\omega_P, *\omega_P \rangle + 2\langle \omega_P, *\omega_P\rangle}.
    \end{align}
    By \thref{norm_one}, we have $\langle \omega_P, \omega_P \rangle = 1$ and $\langle*\omega_P, *\omega_P\rangle = \langle\omega_{P^\perp}, \omega_{P^\perp} \rangle = 1$. By \thref{nontrivial_intersection} we have $\langle\omega_P, *\omega_P \rangle = 0$. So our expression for $\|\pi_+(\omega_P)\|$ reduces to $\frac{1}{2}\sqrt{1 + 1 + 0} = \frac{1}{\sqrt{2}}$. The argument for $\|\pi_-(\omega_P)\|$ is identical.

    To prove the backward inclusion $S^2_- \times S^2_+ \subseteq \im(\omega)$, let $\alpha \in \Lambda^2\R^4$ be a vector for which $\|\pi_-(\alpha)\| = \|\pi_+(\alpha)\| = \frac{1}{\sqrt{2}}$. Note that since $\pi_-(\alpha)$ is orthogonal to $\pi_+(\alpha)$, we have $\|\alpha\| = 1$. Our goal is to show that $\alpha \in \im(\omega)$. By Proposition~\ref{darboux}, we have $\alpha = a\omega_P + b \omega_Q$ where $a,b$ are real numbers and $P, Q$ are planes which intersect trivially. We now simplify the expression $\langle \alpha, *\alpha \rangle$ in two different ways. On one hand, 
    \begin{align}
        \langle \alpha, *\alpha \rangle &= \langle \pi_-(\alpha) + \pi_+(\alpha), \; *\pi_-(\alpha) + *\pi_+(\alpha) \rangle \\
        &= \langle \pi_-(\alpha) + \pi_+(\alpha), \; -\pi_-(\alpha) + \pi_+(\alpha) \rangle \\
        &= -\langle\pi_-(\alpha), \pi_-(\alpha) \rangle + \langle \pi_+(\alpha), \pi_+(\alpha) \rangle \\
        &= -\frac{1}{2} + \frac{1}{2} \\
        &= 0.
    \end{align}
    On the other hand, 
    \begin{align}
        \langle \alpha, *\alpha \rangle &= \langle (a\omega_P + b \omega_Q), *(a\omega_P + b \omega_Q)\rangle
        \\ &= \underbrace{a^2\langle\omega_P, *\omega_P\rangle + b^2\langle\omega_Q, *\omega_Q \rangle}_{= 0 \text{ by \thref{nontrivial_intersection}}} + \underbrace{ab\langle \omega_P, *\omega_Q \rangle + ba \langle \omega_Q, *\omega_P \rangle}_{ = 2ab\langle \omega_P, *\omega_Q \rangle \text{ since $*$ is orthogonal}} \\
        &= 2ab\langle\omega_P, *\omega_Q \rangle.
    \end{align}
    So $2ab\langle\omega_P, *\omega_Q \rangle = 0$. By \thref{nontrivial_intersection} we have $\langle \omega_P, *\omega_Q \rangle \neq 0$, so at least one of $a$ and $b$ must be 0. Assume without loss of generality that $b=0$. We now have $\alpha = a \omega_P$. By \thref{norm_one} we have $\|\omega_P\| = 1$, so the fact that $\|\alpha\| = 1$ forces $a$ to equal $\pm 1$, meaning $\alpha = \pm\omega_P \in \im(\omega)$.
\end{proof}
Here is an illustration summarizing the current situation:
\begin{center}
    \includegraphics[width = 7.8 cm]{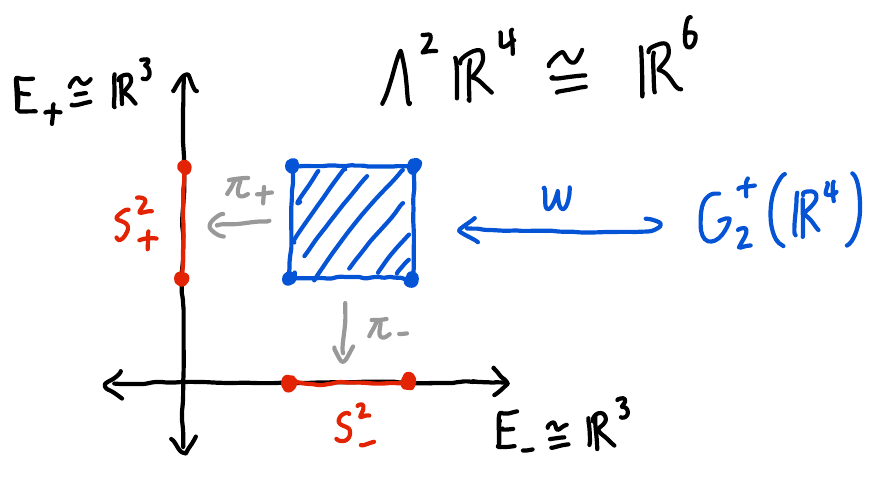}
\end{center}
\begin{corollary}
    The map $G_2^+(\R^4) \to S^2_- \times S^2_+$ given by $P \mapsto \bigl(\pi_-(\omega_P), \pi_+(\omega_P)\bigr)$ is a homeomorphism.
\end{corollary}
\subsection{Great circle fibrations as submanifolds of $G_2^+(\R^4)$}

Let $\mathcal{F}$ denote the space of all fibrations of $S^3$ by oriented great circles, which includes the Hopf fibrations. As we observed at the beginning of this section, an oriented great circle in $S^3$ may be associated with the oriented 2-plane in $\R^4$ that it sits positively inside of. This allows us to associate to each fibration in $\mathcal{F}$ a subset of the Grassmannian:
\begin{definition}[The map $M: \mathcal{F} \to \mathcal{P}(G_2^+(\R^4))$]
    Given a fibration $F$ of $S^3$ by oriented great circles, we define $M_F \subseteq G^+_2(\R^4)$ to be the set of all oriented 2-planes positively containing an oriented great circle in $F$.
\end{definition}

\begin{definition}[Sign of a fibration]\thlabel{sign}
    View $S^3$ as the unit sphere centered at the origin in $\R^4$ and fix a fibration of $S^3$ by oriented great circles. Choose a pair of unit vectors $p, q \in S^3$ not lying on the same fiber. Let $C_p$ and $C_q$ denote the fibers containing $p$ and $q$, respectively. Let $p'$ be $p$ rotated 90 degrees along $C_p$ in the positive direction, and let $q'$ be $q$ rotated 90 degrees along $C_q$ in the positive direction. We define the sign of $F$ to be the sign of the determinant of the $4 \times 4$ matrix $\begin{bmatrix}
        p & p' & q & q'
    \end{bmatrix}$. 
    
    To see that this is independent of our choice of $p$ and $q$, 
    let $X$ denote the space $\{ (p,q) \in S^3 \times S^3 \mid C_p \neq C_q \}$. The fact that $C_p \neq C_q$ tells us that 
    $\det\begin{bmatrix}
        p & p' & q & q'
    \end{bmatrix} \neq 0$ for all $(p,q) \in X$. Since $X$ is path-connected, the intermediate value theorem tells us that any two choices of $(p,q) \in X$ must have the same sign determinant. It's easy to see that for Hopf fibrations, this notion of sign agrees with the one given in \thref{transformation_sign}.
\end{definition}

\begin{definition}[The maps $\theta_+, \theta_-: G_2^+(\R^4) \times G_2^+(\R^4) \to {[0, \pi]}$]\thlabel{theta}
    Given $P, Q \in G_2^+(\R^4)$, we define $\theta_+(P,Q)$ to be the angle between the vectors $\pi_+(\omega_P)$ and $\pi_+(\omega_Q)$, and we define $\theta_-(P,Q)$ to be the angle between the vectors $\pi_-(\omega_P)$ and $\pi_-(\omega_Q)$.
    \begin{center}
        \includegraphics[width = 10 cm]{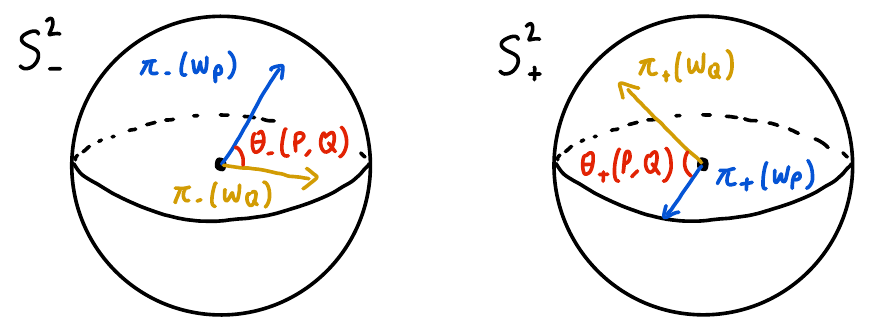}
    \end{center}
\end{definition}

\begin{lemma}\thlabel{theta_nonequal}
\textnormal{\cite[Lemma 5.4]{gluck_warner}}
    Two planes $P, Q \in G_2^+(\R^4)$ intersect nontrivially if and only if $\theta_+(P,Q) = \theta_-(P,Q)$.
\end{lemma}
\begin{proof}
    By \thref{image_of_omega}, all vectors of the form $\pi_+(\omega_P)$ or $\pi_-(\omega_P)$ have constant norm $\frac{1}{\sqrt{2}}$. Therefore, $\theta_+(P,Q) = \theta_-(P,Q)$ if and only if 
    \begin{equation}
    \langle \pi_+(\omega_P), \pi_+(\omega_Q) \rangle = \langle \pi_-(\omega_P), \pi_-(\omega_Q) \rangle.
    \end{equation}
    We compute that 
    \begin{align}
        \langle \pi_+(\omega_P), \pi_+(\omega_Q) \rangle &= \Bigl\langle \frac{\omega_P + *\omega_P}{2}, \frac{\omega_Q + *\omega_Q}{2} \Bigr\rangle \\
        &= \frac{1}{4} \bigl( \underbrace{\langle \omega_P, \omega_Q \rangle + \langle *\omega_P, *\omega_Q \rangle}_{ = 2 \langle \omega_P, \omega_Q \rangle \text{ since $*$ is orthogonal}} + \underbrace{\langle \omega_P, *\omega_Q \rangle + \langle *\omega_P, \omega_Q \rangle}_{= 2\langle \omega_P, *\omega_Q \rangle\text{ since $*$ is orthogonal}} \bigr) \\
        &= \frac{1}{2} \bigl( \langle \omega_P, \omega_Q \rangle + \langle \omega_P, *\omega_Q \rangle \bigr).
    \end{align}
    A similar computation yields
    \begin{equation}
        \langle \pi_-(\omega_P), \pi_-(\omega_Q) \rangle = \frac{1}{2} \bigl( \langle \omega_P, \omega_Q \rangle - \langle \omega_P, *\omega_Q \rangle \bigr).
    \end{equation}
    Therefore, we have
    \begin{equation}
        \langle \pi_+(\omega_P), \pi_+(\omega_Q) \rangle - \langle \pi_-(\omega_P), \pi_-(\omega_Q) \rangle = \langle \omega_P, *\omega_Q \rangle.
    \end{equation}
    By \thref{nontrivial_intersection}, this equals 0 if and only if $P$ and $Q$ intersect nontrivially.
\end{proof}

\begin{lemma}\thlabel{planes_containing_p}
    Fix a point $p \in S^3$, and let $K \subseteq G_2^+(\R^4)$ be the set of oriented 2-planes containing $p$. Then $\omega(K)$ is the graph of an isometry $S^2_- \to S^2_+$.
\end{lemma}
\begin{proof}
    By \thref{theta_nonequal}, $\theta_+(P,Q) = \theta_-(P,Q)$ for all $P,Q \in K$. In particular, the projection map $\pi_-: \omega(K) \to S^2_-$ is injective because
    \begin{align}
        \pi_-(\omega_P) = \pi_-(\omega_Q) &\implies \theta_-(P,Q) = 0 \\
        &\implies \theta_+(P,Q) = \theta_-(P,Q) = 0 \\
        &\implies P = Q.
    \end{align}
    Since choosing an oriented 2-plane containing $p$ is the same as choosing a unit vector orthogonal to $p$, we have that $K \cong S^2$. Since any continuous injection $S^2 \to S^2$ is a bijection, we conclude that $\pi_-$ is a bijection. Therefore, $\omega(K)$ is the graph of the map $\pi_+ \circ \pi_-^{-1}: S^2_- \to S^2_+$, and the fact that $\theta_+(P,Q) = \theta_-(P,Q)$ tells us this map is an isometry.
\end{proof}

\begin{lemma}\thlabel{fibrations_are_spheres}
    Let $F$ be a fibration of $S^3$ by oriented great circles. Then $M_F \cong S^2$.
\end{lemma}
\begin{proof}
    We view $F$ as a quotient map $S^3 \to M_F$ with great circle fibers. $M_F$ is a closed 2-manifold, and since every loop in $M_F$ has a lift over $F$ to a loop in $S^3$, we conclude that $M_F$ has trivial fundamental group. By the classification of surfaces, the only closed 2-manifold with trivial fundamental group is $S^2$.
\end{proof}

\begin{theorem}\thlabel{distance_decreasing}
    A subset $N \subseteq G_2^+(\R^4)$ corresponds to a fibration of $S^3$ by great circles (that is, it is in the image of the map $M$) if and only if $\omega(N)$ is the graph of a distance-decreasing map $S_-^2 \to S_+^2$ or $S_+^2 \to S_-^2$.
\end{theorem}
\begin{proof}
    Forward direction: Let $F$ be a fibration of $S^3$ by oriented great circles. We aim to show that $\omega(M_F) \subseteq S^2_- \times S^2_+$ is the graph of a distance-decreasing map from one $S^2$ to the other. Let $X$ denote the space $\{ (P,Q) \in M_F \times M_F\mid P \neq Q \}$. Since the fibers of $F$ are disjoint, \thref{theta_nonequal} tells us that for all $(P,Q) \in X$ we have $\theta_+(P,Q) - \theta_-(P,Q) \neq 0$. Since $X$ is path-connected, the intermediate value theorem tells us that one of the following hold:
    \begin{enumerate}
        \item $\theta_+(P,Q) - \theta_-(P,Q) < 0$ for all $(P,Q) \in X$.
        \item $\theta_+(P,Q) - \theta_-(P,Q) > 0$ for all $(P,Q) \in X$.
    \end{enumerate}
    Let's assume we're in situation (1), meaning $0 \leq \theta_+(P,Q) < \theta_-(P,Q)$ for all $(P,Q) \in X$. In particular, this means $\theta_-(P,Q) > 0$. By \thref{theta}, this is equivalent to saying $\pi_-(\omega_P) \neq \pi_-(\omega_Q)$ for all $P \neq Q$ in $M_F$, and so the projection map $\pi_-: \omega(M_F) \to S^2_-$ is injective. By \thref{fibrations_are_spheres} we have $\omega(M_F) \cong M_F \cong S^2$, and since any continuous injection $S^2 \to S^2$ is a bijection, we conclude that $\pi_-: \omega(M_F) \to S^2_-$ is a bijection. Therefore, $\omega(M_F)$ is the graph of the map $\pi_+ \circ \pi_-^{-1}: S^2_- \to S^2_+$, and the fact that $\theta_+(P,Q) < \theta_-(P,Q)$ tells us the map is distance-decreasing.

    Backward direction: Let $f: S^2_- \to S_2^+$ be a distance-decreasing map, and suppose $N \subseteq G_2^+(\R^4)$ such that $\omega(N)$ is the graph of $f$. By \thref{theta_nonequal}, $N$ corresponds to a collection of disjoint oriented great circles in $S^3$. To see that these circles cover $S^3$, fix a point $p \in S^3$ and let $K \subseteq G_2^+(\R^4)$ be the set of oriented 2-planes containing $p$. By \thref{planes_containing_p}, $\omega(K)$ is the graph of an isometry $g: S^2_- \to S^2_+$. By the contraction mapping theorem for compact spaces, $g^{-1} \circ f: S^2_- \to S^2_-$ has a fixed point $a \in S^2_-$. So then $a + f(a) = a + g(a) \in \omega(K) \cap \omega(N)$. Therefore, $\omega^{-1}(a + f(a))$ is an oriented great circle in $N$ containing $p$. Since $p$ was arbitrary, we find that the oriented great circles in $N$ cover $S^3$. For details on why these great circles have the local product structure necessary to be a fiber bundle, see \cite[Section 6]{gluck_warner}.
\end{proof}
And so we have proven \thref{gluck_warner}. We now introduce one more theorem from \cite{gluck_warner} that will serve as the starting point for the main idea of this paper. 
\begin{theorem}\thlabel{hopfs_are_slices}
    \textnormal{\cite[Proposition 5.10]{gluck_warner}}
    Let $H$ be a Hopf fibration. 
    \begin{itemize}
        \item If the sign of $H$ is positive, then $\omega(M_H) = S_-^2 \times \{ q \}$ for some $q \in S_+^2$. 
        \item If the sign of $H$ is negative, then $\omega(M_H) = \{p\} \times S_+^2$ for some $p \in S_-^2$.
    \end{itemize}
\end{theorem}
\begin{proof}
    By \thref{transformation_sign}, we may choose an orthonormal basis $\{ f_1, f_2, f_3, f_4\}$ under which $H$ looks like the standard Hopf fibration (meaning the planes $\lspan(f_1, f_2)$ and $\lspan(f_3, f_4)$ are in $M_H$), with the orientation of this basis depending on the sign of $H$.

    \thref{standard_hopf} tells us that if we identify the basis $\{f_1, f_2, f_3, f_4\} \subseteq \R^4$ with 
    \begin{equation}
        \{ (1,0), (i,0), (0,1), (0,i) \} \subseteq \C^2,
    \end{equation}
    then all planes in $M_H$ are of the form $\lspan_\R(x, ix)$ for $x \in \C^2$. Writing this in terms of the $f_i$, we find that all planes in $M_H$ are of the form
    \begin{equation}
        P = \lspan(af_1 + bf_2 + cf_3 + df_4, \; -bf_1 + af_2 - df_3 + cf_4).
    \end{equation}
    Assuming $a^2 + b^2 + c^2 + d^2 = 1$, we then compute that 
    \begin{align}
        \omega_P = (a^2 + b^2) f_1^* \wedge f_2^* + (-ad+bc)f_1^* \wedge f_3^* + (ac + bd) f_1^* &\wedge f_4^* \\
        + (-bd - ac) f_2^* \wedge f_3^* + (bc-ad) f_2^* &\wedge f_4^* \\
        + (c^2 + d^2) f_3^* &\wedge f_4^*.
    \end{align}
    If the sign of $H$ is positive, then $\{f_1, f_2, f_3, f_4\}$ is a positively oriented basis of $\R^4$, and so
    \begin{align}
        *\omega_P = (a^2 + b^2) f_3^* \wedge f_4 ^* + (ad-bc)f_2^* \wedge f_4^* \phantom{-} + (ac + bd) f_2^* &\wedge f_3^* \\
        + (-bd - ac) f_1^* \wedge f_4^* + (-bc+ad) f_1^* &\wedge f_3^* \\
        + (c^2 + d^2) f_1^* &\wedge f_2^*.
    \end{align}
    Therefore, 
    \begin{align}
        \pi_+(\omega_P) &= \frac{\omega_P + *\omega_P}{2}\\
        &= \frac{(a^2 + b^2 + c^2 + d^2) f_1^* \wedge f_2^* + (a^2 + b^2 + c^2 + d^2) f_3^* \wedge f_4^*}{2} \\
        &= \frac{f_1^* \wedge f_2^* + f_3^* \wedge f_4^*}{2}.
    \end{align}
    In particular, this does not depend on the values of $a,b,c,d$. Letting $q = \frac{1}{2}(f_1^* \wedge f_2^* + f_3^* \wedge f_4^*)$, we conclude that $\pi_+ \circ \omega(M_H) = \{q\}$, meaning $\omega(M_H) = S^2_- \times \{q\}$, as desired.

    If the sign of $H$ is negative, then $*\omega_P$ is negative what is written above. A similar argument reveals that $\pi_-(\omega_P) = q$ and $\omega(M_H) = \{q\} \times S^2_+$.
\end{proof}

\begin{corollary}\thlabel{neg_to_pos}
    Distance-decreasing maps $S^2_- \to S^2_+$ correspond to positive fibrations, while distance-decreasing maps $S^2_+ \to S^2_-$ correspond to negative fibrations.
\end{corollary}
\begin{proof}
    Any distance-decreasing map $S^2_- \to S^2_+$ may be continuously deformed through the space of distance-decreasing maps into a constant map, and the graph of a constant map $S^2_- \to S^2_+$ corresponds to a positive Hopf fibration by \thref{hopfs_are_slices}. The negative case follows by the same argument.
\end{proof}

\section{Idea for generalizing}\label{motivation}

\thref{hopfs_are_slices} has an interesting consequence: If $H^+$ is a positive Hopf fibration and $H^-$ is a negative Hopf fibration, then the theorem tells us that $M_{H^+} = S^2_- \times \{q\}$ for some $q \in S^2_+$, and $M_{H^-} = \{p\} \times S^2_+$ for some $p \in S^2_-$. Therefore, $M_{H^+}\cap M_{H^-} = \{(p,q)\}$. In other words, 
\begin{fact}\thlabel{opp_sign_share}
    Any two Hopf fibrations of opposite sign are guaranteed to share a unique oriented great circle.
\end{fact}
\noindent And similarly,
\begin{fact}\thlabel{same_sign_avoid}
    Any two Hopf fibrations of the same sign are either equal, or share no oriented great circles.
\end{fact}
These observations allow us to give an alternate characterization of the homeomorphism $S^2 \times S^2 \to G_2^+(\R^4)$:

\begin{definition}[The map $\varphi_p: S^2 \times S^2 \to G_2^+(\R^4)$]\thlabel{hopf_homeo}
    Fix a point $p \in S^3$, and identify $S^2$ with the set $S^3 \cap \lspan(p)^\perp$ of unit vectors orthogonal to $p$. Given vectors $u, v \in S^2$, let $H^-$ be the negative Hopf fibration containing $\lspan(p,u)$ and $H^+$ be the positive Hopf fibration containing $\lspan(p,v)$. ($H^-$ and $H^+$ are unique by \thref{same_sign_avoid}).
    \begin{center}
        \includegraphics[width = 7.4 cm]{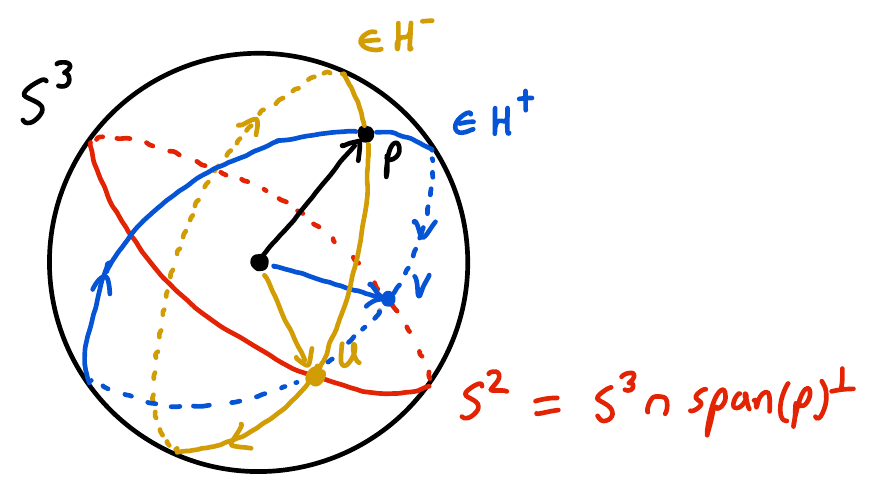}
    \end{center}
    We define $\varphi_p(u, v)$ to equal the oriented 2-plane containing the circle of agreement of $H^-$ and $H^+$. (This circle of agreement exists and is unique by \thref{opp_sign_share}).
\end{definition}

\begin{definition}[The maps $\psi_{p-}: S^2 \to S^2_-$ and $\psi_{p+}: S^2 \to S^2_+$]\thlabel{psi}
    Fix $p \in S^3$ and identify $S^2$ with $S^3 \cap \lspan(p)^\perp$. Given $u \in S^2$, we define 
    \begin{equation}
        \psi_{p-}(u) = \pi_-(\omega_{\lspan(p,u)}), \qquad \psi_{p+}(u) = \pi_+(\omega_{\lspan(p,u)}).
    \end{equation}
\end{definition}

\begin{proposition}\thlabel{S2_isometry}
    If we rescale $S^2_-$ and $S^2_+$ to have radius 1, then the maps $\psi_{p-}$ and $\psi_{p+}$ are isometries.
\end{proposition}
\begin{proof}
    $\psi_{p+}(v) = \pi_+(\omega_{\lspan(p,v)}) = \pi_+(p^* \wedge v^*).$
    The map $v \mapsto \sqrt{2} \cdot \pi_+(p^* \wedge v^*)$ is linear when viewed as a map $\R^3 \to \R^3$, and it's norm-preserving since it takes $S^2$ to $\sqrt{2} \cdot S^2_+$. Therefore, it's an isometry. The $\psi_{p-}$ case follows by the same argument.
\end{proof}
\begin{proposition}\thlabel{homeomorphism_compatibility}
    The map $\varphi_p: S^2 \times S^2 \to G_2^+(\R^4)$ given by \thref{hopf_homeo} and the homeomorphism $\omega: G_2^+(\R^4) \to S^2_- \times S^2_+$ given by \thref{image_of_omega} are equivalent in the sense that 
    \begin{equation}
    \omega \circ \varphi_p(u,v) = \psi_{p-}(u) + \psi_{p+}(v).
    \end{equation}
    In particular, the map $\omega \circ \varphi_p$ is an isometry (up to scaling), and $\varphi_p$ is a homeomorphism.
\end{proposition}
\begin{proof}
    First we show that $\pi_- \circ \omega \circ \varphi_p(u,v)$ does not depend on $v$. By \thref{hopf_homeo}, as $v$ varies over $S^2$, we have that $\varphi_p(u,v)$ varies over $M_{H^-}$, where $H^-$ is the negative Hopf fibration containing $\lspan(p,u)$. And by \thref{hopfs_are_slices}, as $\varphi_p(u,v)$ varies over $M_{H^-}$, the point $\pi_- \circ \omega \circ \varphi_p(u,v)$ stays constant. 
    A similar argument shows that $\pi_+ \circ \omega \circ \varphi_p(u,v)$ doesn't depend on $u$. So then 
    \begin{align}
        \omega \circ \varphi_p(u,v) &= \pi_- \circ \omega \circ \varphi_p(u,v) + \pi_+ \circ \omega \circ \varphi_p(u,v) \\
        &= \pi_- \circ \omega \circ \varphi_p(u,u) + \pi_+ \circ \omega \circ \varphi_p(v,v) \\
        &= \pi_-(\omega_{\lspan(p,u)}) + \pi_+(\omega_{\lspan(p,v)}) \\
        &= \psi_{p-}(u) + \psi_{p+}(v),
    \end{align}
    and the proof is complete.
\end{proof}

We can now restate \thref{distance_decreasing} in the language of the map $\varphi_p$:
\begin{theorem}\thlabel{new_formulation}
    Fix $p \in S^3$ and identify $S^2$ with $S^3 \cap \lspan(p)^\perp$. A subset of $G_2^+(\R^4)$ corresponds to a fibration of $S^3$ by great circles if and only if it's of the form $\bigl\{ \varphi_p\bigl(v, f(v)\bigr) \mid v \in S^2 \bigr\}$ or $\bigl\{ \varphi_p\bigl(f(v), v\bigr) \mid v \in S^2 \bigr\}$, where $f: S^2 \to S^2$ is a distance-decreasing map.
\end{theorem}
\begin{proof}
    Forward direction: Let $F$ be a fibration of $S^3$ by oriented great circles. By \thref{distance_decreasing} we have that $\omega(M_F)$ is the graph of a distance-decreasing map $S^2_- \to S^2_+$ or $S^2_+ \to S^2_-$. Let's assume the map is $S^2_- \to S^2_+$, and let's call it $g$. Then we can write $M_F = \{ \omega^{-1}(a + g(a)) \mid a \in S^2_- \}$. We will show that 
    \begin{equation}
        \{ \omega^{-1}(a + g(a)) \mid a \in S^2_- \} = \{ \varphi_p(v, \, \psi^{-1}_{p+} \circ g \circ \psi_{p-}(v)) \mid v \in S^2 \}.
    \end{equation}
    The map $\psi^{-1}_{p+} \circ g \circ \psi_{p-}: S^2 \to S^2$ is distance-decreasing by \thref{S2_isometry}, so if we can prove the above equality then we're done. By \thref{homeomorphism_compatibility} we have 
    \begin{equation}
        \{ \omega \circ \varphi_p(v, \, \psi^{-1}_{p+} \circ g \circ \psi_{p-}(v)) \mid v \in S^2 \} = \{ \psi_{p-}(v) + g \circ \psi_{p-}(v) \mid v \in S^2\}.
    \end{equation}
    Substituting $a = \psi_{p-}(v)$, we get 
    \begin{equation}
        \{a + g(a) \mid a \in S^2_-\},
    \end{equation}
    as desired. 

    Backward direction: Let $f: S^2 \to S^2$ be a distance-decreasing map, and let $N = \{ \varphi_p(v, f(v)) \mid v \in S^2 \} \subseteq G_2^+(\R^4)$. Our goal is to show that $N$ corresponds to a fibration of $S^3$ by great circles. By \thref{distance_decreasing} it suffices to show that $\omega(N)$ is the graph of a distance-decreasing map $S^2_- \to S^2_+$. We will show that
    \begin{equation}
        \omega(N) = \{ a + \psi_{p+} \circ f \circ \psi_{p-}^{-1}(a) \mid a \in S^2_- \}.
    \end{equation}
    The map $\psi_{p+} \circ f \circ \psi_{p-}^{-1}$ is distance-decreasing by \thref{S2_isometry}, so if we can prove the above equality then we're done. By \thref{homeomorphism_compatibility} we have 
    \begin{equation}
        \omega(N) = \{ \omega \circ \varphi_p(v, f(v)) \mid v \in S^2 \} = \{ \psi_{p-}(v) + \psi_{p+} \circ f(v) \mid v \in S^2 \}.
    \end{equation}
    Substituting $a = \psi_{p-}(v)$, we get 
    \begin{equation}
        \{ a + \psi_{p+} \circ f \circ \psi_{p-}^{-1}(a) \mid a \in S^2_- \},
    \end{equation}
    as desired.
\end{proof}

The statement of \thref{new_formulation} gives us hope for generalizing \thref{gluck_warner}. In their original paper, Gluck and Warner proved \thref{gluck_warner} via \thref{distance_decreasing}, which is formulated in terms of the decomposition of $G_2^+(\R^4)$ into a product of spheres, a phenomenon unique to $\R^4$. On the other hand, \thref{new_formulation} is formulated purely in terms of circle agreement of Hopf fibrations of opposite sign (\thref{opp_sign_share}). This is the motivation for our main result \thref{main_result}, which shows that existence of a shared circle has an analogue in all dimensions. The main obstacle to generalizing Gluck and Warner's result then becomes uniqueness, which as we'll see in Section~\ref{next_steps} does not hold in general.

\section{Hopf fibrations are linear}\label{hopf_linear}
Let $\mathcal{F}$ denote the set of fibrations of $S^{2n-1}$ by oriented great circles, and let $\operatorname{Map}(\R^{2n}, \R^{2n})$ denote the set of functions $\R^{2n} \to \R^{2n}$.
\begin{definition}[The map $I: \mathcal{F} \to \operatorname{Map}(\R^{2n}, \R^{2n})$]\thlabel{the_map_I}
    Let $F$ be a fibration of $S^{2n-1}$ by oriented great circles. We define $I_F$ to be the map $\R^{2n} \to \R^{2n}$ which sends each $p \in \R^{2n}$ to $p$ rotated 90 degrees in the positive direction along the oriented 2-plane containing $p$. It's clear that $I$ is injective.
    \begin{center}
        \includegraphics[width = 4.1 cm]{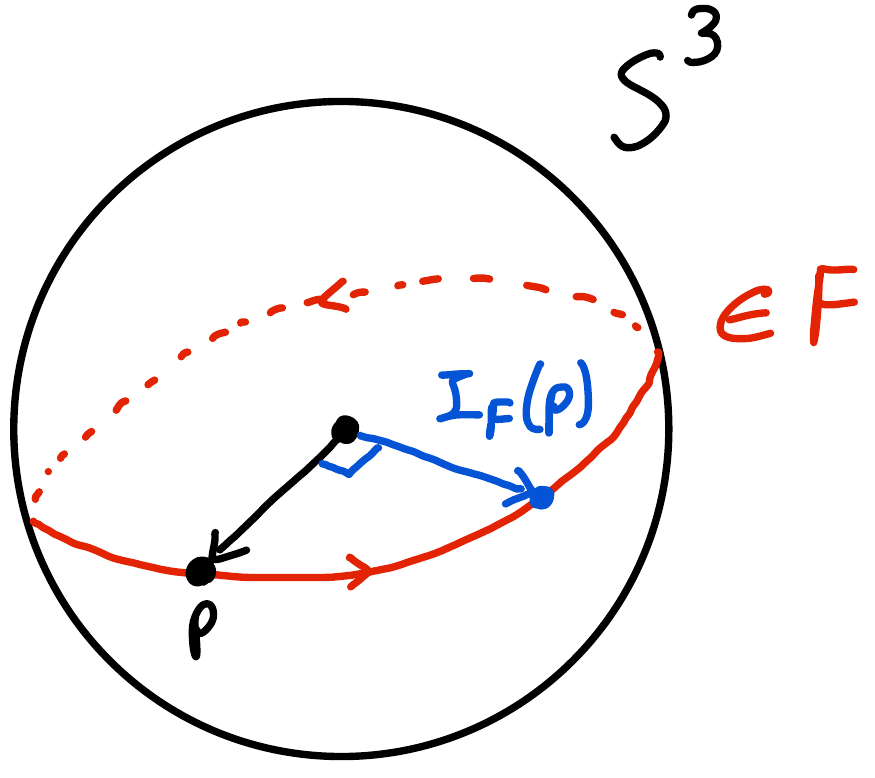}
    \end{center}
\end{definition}
The purpose of this section is to prove \thref{correspondence} which states that $I$ restricts to a bijection between Hopf fibrations and a type of linear transformation called an \emph{orthogonal complex structure}. This will ultimately allow us to reduce the problem of showing two Hopf fibrations agree on a fiber to one of showing a certain linear map has nontrivial kernel.

\begin{definition}[Complex structure]\thlabel{complex_structure}
    A complex structure on $\R^{2n}$ is a $2n \times 2n$ matrix $J$ with the property that $J^2 = -\id$.
\end{definition}

\begin{lemma}\thlabel{complexStructureEquiv}
    Let $J$ be a square matrix. Then any two of the following implies the third: 
    \begin{enumerate}
        \item $J$ is a complex structure.
        \item $J$ is skew-symmetric.
        \item $J$ is orthogonal.
    \end{enumerate}
\end{lemma}
\begin{proof}
    Let
    \begin{align}
        A &= J^2 + \id, \\
        B &= J + J^T, \\
        C &= JJ^T - \id.
    \end{align}
    Conditions (1), (2), and (3) in the lemma are saying that $A$, $B$, and $C$ are 0, respectively. The claim then follows from the fact that $A = JB-C$, $B = AJ^T - JC$, and $C = JB-A$.
\end{proof}

\begin{definition}[Orthogonal complex structure]\thlabel{orthogonal_complex_structure}
    A matrix satisfying the hypotheses of \thref{complexStructureEquiv} is called an orthogonal complex structure.
\end{definition}

\begin{lemma}\thlabel{ocs_conjugation}
    Let $J: \R^{2n} \to \R^{2n}$ be an orthogonal complex structure, and let $T: \R^{2n} \to \R^{2n}$ be an orthogonal transformation. Then $TJT^{-1}$ is an orthogonal complex structure.
\end{lemma}
\begin{proof}
    $TJT^{-1}$ is orthogonal since $J$ and $T$ are orthogonal, and it's a complex structure since $(TJT^{-1})^2 = TJ^2T^{-1} = -TT^{-1} = -\id$.
\end{proof}

\begin{lemma}\thlabel{conjugation}
    Let $F$ be a fibration of $S^{2n-1}$ by oriented great circles, and let $T: \R^{2n} \to \R^{2n}$ be an orthogonal transformation. Then $I_{T(F)} = TI_FT^{-1}$.
\end{lemma}
\begin{proof}
    Fix $v \in \R^{2n}$ and let $P$ be the oriented 2-plane in $F$ containing $v$. Then $T(P)$ is the oriented 2-plane in $T(F)$ containing $T(v)$. We observe that 
    \begin{align}
        I_{T(F)}(T(v)) &= [\text{$T(v)$ rotated 90 degrees along $T(P)$}] \\
        &= T([\text{$v$ rotated 90 degrees along $P$}]) \\
        &= T(I_F(v)).
    \end{align}
    Substituting $u = T(v)$, we get $I_{T(F)}(u) = T(I_F(T^{-1}(u)))$. Since $v$ was arbitrary, we conclude that $I_{T(F)} = TI_FT^{-1}$.
\end{proof}

\begin{definition}[Standard complex structure]\thlabel{standard_complex_structure}
    Let $H_0$ be the standard Hopf fibration on $S^{2n-1}$. Then identifying $\R^{2n}$ with $\C^n$, we have that $I_{H_0}: \R^{2n} \to \R^{2n}$ is the map given by multiplication by $i$. In other words, $I_{H_0}$ is the following $2n \times 2n$ real matrix:
    \begin{equation}
        \begin{bmatrix}
            0 & -1 &  &  &  \\
            1 & 0 &  &  &  \\
             &  & \ddots &  &  \\
             &  &  & 0 & -1 \\
             &  &  & 1 & 0 \\
        \end{bmatrix}.
    \end{equation}
    We call this the standard complex structure on $\R^{2n}$. 
\end{definition}

\begin{proposition}\thlabel{hopf_implies_ocs}
    Let $H$ be a Hopf fibration. Then $I_H$ is an orthogonal complex structure.
\end{proposition}
\begin{proof}
    By \thref{transformation_sign} we have $H = T(H_0)$, where $H_0$ is the standard Hopf fibration and $T: \R^{2n} \to \R^{2n}$ is an orthogonal map. By \thref{standard_complex_structure}, $I_{H_0}$ is an orthogonal complex structure, and by \thref{ocs_conjugation}, $TI_{H_0}T^{-1}$ is one as well. By \thref{conjugation} we have $TI_{H_0}T^{-1} = I_H$, and we're done. 
\end{proof}


Let $\mathcal{H} \subseteq \mathcal{F}$ be the set of Hopf fibrations on $S^{2n-1}$. \thref{hopf_implies_ocs} tells us that $\im(I|_\mathcal{H})$ is contained in the set of orthogonal complex structures. The remainder of this section will be devoted to proving \thref{ocs_implies_hopf}, which gives the other direction: The set of orthogonal complex structures is contained in $\im(I|_\mathcal{H})$.

\begin{lemma}\thlabel{Iinvariant}
    Let $J$ be a complex structure on $\R^{2n}$, and let $v \in \R^{2n}$ be a nonzero vector. Then $\lspan\{v, Jv\}$ is a 2-dimensional $J$-invariant subspace.
\end{lemma}
\begin{proof}
    $J(\lspan\{v, Jv\}) = \lspan\{Jv, J^2v\} = \lspan\{Jv, -v\} = \lspan\{v, Jv\}.$
\end{proof}


\begin{lemma}\thlabel{orthogonalDecomposition}
    Let $J$ be a complex structure on $\R^{2n}$. Then we may express $\R^{2n}$ as a direct sum of $J$-invariant 2-dimensional subspaces. If $J$ is orthogonal, then these subspaces may be chosen to be mutually orthogonal.
\end{lemma}
\begin{proof}
    We construct our subspaces inductively. Suppose we have already chosen a set of 2-dimensional $J$-invariant subspaces $\{W_1, \cdots, W_k\}$. If $k = n$ then we are done. Otherwise, let $v$ be a nonzero vector not in $W_1 + \cdots + W_k$, and let $W_{k+1} = \lspan\{v, Jv\}$, which is $J$-invariant by \thref{Iinvariant}. It remains to show that $W_{k+1}$ and $W_1 + \cdots + W_k$ intersect trivially. Suppose otherwise. Then the intersection would be a 1-dimensional $J$-invariant subspace, meaning $J$ has an eigenvalue $\lambda \in \R$. Since $J^2 = -\id$, we have $\lambda^2 = -1$, a contradiction.
    
    In the case that $J$ is orthogonal, the argument is essentially the same: Let $v$ be a nonzero vector orthogonal to $W_1 + \cdots + W_k$, and let $W_{k+1} = \lspan\{v, Jv\}$, which is $J$-invariant by \thref{Iinvariant}. Since $J$ is an orthogonal transformation, $Jv$ is orthogonal to $J(W_1 + \cdots + W_k) = W_1+ \cdots + W_k$. Therefore, $W_{k+1}$ is a 2-dimensional $J$-invariant subspace orthogonal to $W_1 + \cdots + W_k$.
\end{proof}

\begin{remark}
    The above lemma shows that $\R^n$ admits a complex structure if and only if $n$ is even.
\end{remark}

\begin{lemma}\thlabel{conjugate_to_standard}
    All orthogonal complex structures are orthogonally conjugate to the standard complex structure.
\end{lemma}
\begin{proof}
    Let $J$ be an orthogonal complex structure on $\R^{2n}$. To show that $J$ is orthogonally conjugate to the standard complex structure, it suffices to construct an orthonormal basis $\{e_1, \cdots, e_{2n}\}$ such that for all $k$ from 1 to $n$ we have $J(e_{2k-1}) = e_{2k}$ and $J(e_{2k}) = -e_{2k-1}$. 

    By \thref{orthogonalDecomposition}, we can write $\R^{2n} = W_1 \oplus \cdots \oplus W_n$, where each $W_k$ is a 2-dimensional $J$-invariant subspace, and the $W_k$ are mutually orthogonal. For all $k$ from 1 to $n$, let $e_{2k-1} \in W_k$ be an arbitrary unit vector, and let $e_{2k} = J(e_{2k-1})$. Since $J$ is an orthogonal transformation, $e_{2k}$ is a unit vector. Since $J$ is skew-symmetric, $e_{2k}$ is orthogonal to $e_{2k-1}$. Since $J^2 = -\id$, we have $J(e_{2k}) = -e_{2k-1}$. Finally, the set of vectors $\{e_1, \cdots, e_{2n}\}$ is orthonormal because the $W_k$ are mutually orthogonal and $\{e_{2k-1}, e_{2k}\}$ is an orthonormal basis for $W_k$ for each $k$.
\end{proof}

\begin{proposition}\thlabel{ocs_implies_hopf}
    Let $J$ be an orthogonal complex structure. Then there exists a Hopf fibration $H$ such that $I_H = J$.
\end{proposition}
\begin{proof}
    By \thref{conjugate_to_standard} we can write $J = TI_{H_0}T^{-1}$, where $H_0$ is the standard complex structure and $T$ is an orthogonal transformation. By \thref{conjugation}, this equals $I_{T(H_0)}$. Letting $H = T(H_0)$ completes the proof.
\end{proof}

\begin{theorem}\thlabel{correspondence}
    The map $I: \mathcal{F} \to \operatorname{Map}(\R^{2n}, \R^{2n})$ restricts to a bijection between the set of Hopf fibrations of $S^{2n-1}$ and the set of orthogonal complex structures on $\R^{2n}$.
\end{theorem}
\begin{proof}
    We observed in \thref{the_map_I} that $I$ is injective. By Propositions \ref{hopf_implies_ocs} and \ref{ocs_implies_hopf}, the image of $I|_\mathcal{H}$ is equal to the set of orthogonal complex structures, and thus we get the desired bijection.
\end{proof}


\section{Agreement of Hopf fibrations on fibers}\label{hopf_agreement}
The correspondence established in \thref{correspondence} gives rise to the following proposition:
\begin{proposition}\thlabel{win_condition}
    Let $H_1$ and $H_2$ be Hopf fibrations on $S^{2n-1}$. Then $H_1$ and $H_2$ share an oriented fiber if and only if $I_{H_1} - I_{H_2}$ has nontrivial kernel.
\end{proposition}
\begin{proof}
    Forward direction: Assume that $H_1$ and $H_2$ share an oriented great circle $C \subseteq S^{2n-1} \subseteq \R^{2n}$. Let $p \in C$ be an arbitrary unit vector. By \thref{the_map_I} we have
    \begin{equation}
        I_{H_1}(p) = I_{H_2}(p) = [\text{$p$ rotated 90 degrees along $C$}].
    \end{equation}
    This means $(I_{H_1} - I_{H_2})(p) = 0$, and therefore $p \in \ker(I_{H_1} - I_{H_2})$.

    Backward direction: Assume that $\ker(I_{H_1} - I_{H_2})$ is nontrivial. Choose a unit vector $p$ for which $(I_{H_1} - I_{H_2})(p) = 0$. Let $C_1$ and $C_2$ be the oriented great circles in $H_1$ and $H_2$ containing $p$, respectively. Since $I_{H_1}(p) = I_{H_2}(p)$, \thref{the_map_I} tells us that
    \begin{equation}
        [\text{$p$ rotated 90 degrees along $C_1$}] = [\text{$p$ rotated 90 degrees along $C_2$}].
    \end{equation}
    It follows that $C_1 = C_2$, and we're done.
\end{proof}
In this section the goal is to prove our main result, \thref{main_result}. The strategy is to use \thref{win_condition} to translate the result from a fact about Hopf fibrations into a fact about orthogonal complex structures, which we prove in \thref{oppositeCircles} using linear algebra.

\begin{definition}[Sign of a complex structure]\thlabel{signCheck}
    Let $J$ be a complex structure on $\R^{2n}$. By \thref{orthogonalDecomposition} we can write $\R^{2n} = W_1 \oplus \cdots \oplus W_n$, where each $W_k$ is a 2-dimensional $J$-invariant subspace. For each $k \in \{1, \cdots, n\}$, choose a nonzero vector $e_k \in W_k$, and let $Q$ be the $2n \times 2n$ matrix given by 
    \begin{equation}
        Q = \begin{bmatrix}
            e_1 & Je_1 & \cdots & e_n & Je_n
        \end{bmatrix},
    \end{equation}
    where each entry is viewed as a column vector. Then we define the sign of $J$ to equal the sign of the determinant of $Q$. It's easy to see that this notion of sign is compatible with the one given in \thref{sign} in the sense that if $H$ is a Hopf fibration, then $H$ and $I_H$ have the same sign.
\end{definition}

\begin{proposition}\thlabel{signOfNegative}
    Let $J$ be a complex structure on $\R^{2n}$. If $n$ is even, then $J$ and $-J$ have the same sign. If $n$ is odd, then $J$ and $-J$ have opposite signs.
\end{proposition}
\begin{proof}
    When $J$ is replaced by $-J$, the matrix $Q$ from \thref{signCheck} has $n$ of its columns multiplied by $-1$. This changes the sign of $\det(Q)$ if and only if $n$ is odd.
\end{proof}

\begin{lemma}\thlabel{basisChoice}
    Let $J$ and $K$ be orthogonal complex structures on $\R^{2n}$, and let $p \in \R^{2n}$ be a unit vector. Then there exists a pair of orthonormal bases $\mathcal{E} = \{e_1, \cdots, e_{2n}\}$ and $\mathcal{F} = \{f_1, \cdots, f_{2n}\}$ such that
    \begin{enumerate}
        \item $e_1 = f_1 = p$.
        \item $J$ is the standard complex structure in $\mathcal{E}$. 
        \item $K$ is the standard complex structure in $\mathcal{F}$.
        \item The change of basis matrix $Q$ from $\mathcal{E}$ to $\mathcal{F}$ (that is, $f_j = \sum_i Q_{ij} e_i$) is of the form
            \begin{equation}
                Q = \begin{bmatrix}
                    1 & & & & & & \\
                      & c_1 & -s_1 \\
                      & s_1 & c_1 \\
                      & & & \ddots \\
                      & & & & c_{n-1} & -s_{n-1} \\
                      & & & & s_{n-1} & c_{n-1} \\
                      & & & & & & \pm1
                \end{bmatrix},
            \end{equation}
            where 
            \begin{itemize}
                \item $c_i^2 + s_i^2 = 1$ for all $i \in \{1, \cdots, n-1\}$.
                \item The $\pm1$ in the lower right corner is positive when $J$ and $K$ have the same sign, and negative when $J$ and $K$ have opposite signs. 
            \end{itemize}
    \end{enumerate}
\end{lemma}
\begin{proof}
    We build $\mathcal{E}$ and $\mathcal{F}$ inductively. First, let $e_1 = f_1 = p$. Now, suppose that we have already chosen $\{e_1, \cdots, e_{2k-1}\}$ and $\{f_1, \cdots, f_{2k-1}\}$ to be orthonormal sets of vectors such that 
    \begin{enumerate}[label=\alph*)]
        \item $\{e_1, \cdots, e_{2k-1}\}$ and $\{f_1, \cdots, f_{2k-1}\}$ span the same subspace, call it $W$.
        \item The change of basis matrix from $\{e_1, \cdots, e_{2k-1}\}$ to $\{f_1, \cdots, f_{2k-1}\}$ looks like the top left $(2k-1) \times (2k-1)$ submatrix of $Q$.
        \item For all $i \in \{1, \cdots, k-1\}$ we have $e_{2i} = J(e_{2i-1})$ and $f_{2i} = K(f_{2i-1})$. 
    \end{enumerate}
    Our goal now is to choose $e_{2k}$, $f_{2k}$, $e_{2k+1}$, $f_{2k+1}$ such that (a), (b), and (c) still hold.
    \begin{center}
        \includegraphics[width = 15 cm]{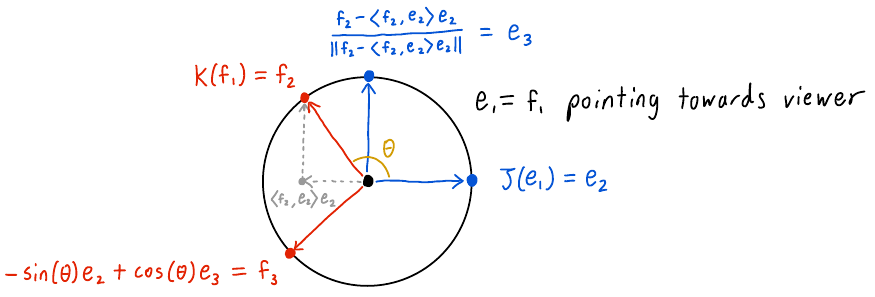}
    \end{center}
    To ensure that (c) still holds, we let $e_{2k} = J(e_{2k-1})$ and $f_{2k} = K(f_{2k-1})$. Since $J$ is skew-symmetric, $e_{2k}$ is orthogonal to $e_{2k-1}$. We have by assumption that $e_{2k-1}$ is orthogonal to $\lspan\{ e_1, \cdots, e_{2k-2} \}$, which by (c) is a $J$-invariant subspace. Since $J$ is orthogonal, this implies that $J(e_{2k-1}) = e_{2k}$ is a unit vector orthogonal to $J(\lspan\{ e_1, \cdots, e_{2k-2}\}) = \lspan\{ e_1, \cdots, e_{2k-2} \}$. We conclude that $e_{2k}$ is orthogonal to $W$. By the same argument we have that $f_{2k}$ is orthogonal to $W$. 
    
    At this point, if $k = n$, then the inductive process terminates. Since we chose $e_{2n}$ and $f_{2n}$ to be unit vectors orthogonal to $W$, which has codimension 1, it follows that $f_{2n} = \pm e_{2n}$. The sign is determined by the sign of the determinant of $Q$, which must be positive when $J$ and $K$ have the same sign, and negative when they have opposite signs.
    
    So from now on assume $k < n$. We split into cases:
    \begin{itemize}
        \item Case 1: If $e_{2k}$ and $f_{2k}$ are linearly independent, choose $e_{2k+1}$ to be the component of $f_{2k}$ perpendicular to $e_{2k}$, rescaled to have norm 1. 

        To show that our choice is valid, we must show that the vector $e_{2k+1}$ is orthogonal to $\lspan\{e_1, \cdots, e_{2k}\}$:
        \begin{itemize}
            \item It's orthogonal to $e_{2k}$ by construction. 
            \item It's orthogonal to $W = \lspan\{e_1, \cdots, e_{2k-1}\}$ since it's a linear combination of $e_{2k}$ and $f_{2k}$, both of which are orthogonal to $W$.
        \end{itemize}
        \item Case 2: If $e_{2k}$ and $f_{2k}$ are linearly dependent, choose $e_{2k+1}$ to be any unit vector orthogonal to $\lspan\{e_1, \cdots, e_{2k}\}$.
    \end{itemize}
    In either of the above cases, we have $f_{2k} = \cos(\theta) e_{2k} + \sin(\theta)e_{2k+1}$ for some $\theta \in [0, \pi]$. Letting $c_k = \cos(\theta)$ and $s_k = \sin(\theta)$, we get 
    \begin{equation}\label{eq1}
        f_{2k} = c_k e_{2k} + s_ke_{2k+1}.
    \end{equation}
    Finally, let 
        \begin{equation}\label{eq2}
            f_{2k+1} = -s_k e_{2k} + c_k e_{2k+1}.
        \end{equation}
        To show that this choice is valid, we must show that $f_{2k+1}$ is orthogonal to $\lspan\{f_1, \cdots, f_{2k}\}$:
    \begin{itemize}
        \item[]
        \begin{itemize}
            \item It's orthogonal to $f_{2k}$ since $\langle f_{2k}, f_{2k+1}\rangle = -c_k s_k + c_k s_k = 0$.
            \item It's orthogonal to $W = \lspan\{f_1, \cdots, f_{2k-1}\}$ since it's a linear combination of $e_{2k}$ and $e_{2k+1}$, both of which are orthogonal to $W$. 
        \end{itemize}
    \end{itemize}
    The relations \eqref{eq1} and \eqref{eq2} tell us that conditions (a) and (b) still hold for $\{e_1, \cdots, e_{2k+1}\}$ and $\{f_1, \cdots, f_{2k+1}\}$, completing the proof.
\end{proof}

\begin{lemma}\thlabel{base_case}
    Let $J$ and $K$ be orthogonal complex structures of opposite sign on $\R^4$. Then $J + K$ has nontrivial kernel.
\end{lemma}
\begin{proof}
    This is essentially a restatement of \thref{opp_sign_share} in the language of orthogonal complex structures. Here we give a purely algebraic proof:

    Let $I_0$ denote the standard complex structure. By \thref{basisChoice}, we can choose bases $\mathcal{E} = \{e_1, e_2, e_3, e_4\}$ and $\mathcal{F} = \{f_1, f_2, f_3, f_4\}$ such that $J = I_0$ in $\mathcal{E}$, $K = I_0$ in $\mathcal{F}$, and the change of basis matrix from $\mathcal{E}$ to $\mathcal{F}$ is of the form
    \begin{equation}
        Q = \begin{bmatrix}
            1 & & & \\
            & c & -s & \\
            & s & c & \\
            & & & -1
        \end{bmatrix},
    \end{equation}
    where $c^2 + s^2 = 1$.

    Expressing everything in terms of the basis $\mathcal{E}$, we have $J = I_0$ and $K = Q I_0 Q^{-1}$. We observe that
    \begin{align}
        \ker(J + K) \neq \{0\} &\iff \ker(I_0 + QI_0Q^{-1}) \neq \{0\} \\
        &\iff \det(I_0+QI_0Q^{-1}) = 0 \\
        &\iff \det(I_0Q + QI_0) = 0,
    \end{align}
    so our goal is now to prove $\det(I_0Q + QI_0) = 0$. We compute that 
    \begin{equation}
        I_0Q = \begin{bmatrix}
            0 & -1 & & \\
            1 & 0 & & \\
            & & 0 & -1 \\
            & & 1 & 0
        \end{bmatrix}
        \begin{bmatrix}
            1 & & & \\
            & c & -s & \\
            & s & c & \\
            & & & -1
        \end{bmatrix} = \begin{bmatrix}
            & -c & s & \\
            1 & & & \\
            & & & 1 \\
            & s & c &
        \end{bmatrix}
    \end{equation}
    and 
    \begin{equation}
        QI_0 = 
        \begin{bmatrix}
            1 & & & \\
            & c & -s & \\
            & s & c & \\
            & & & -1
        \end{bmatrix} 
        \begin{bmatrix}
            0 & -1 & & \\
            1 & 0 & & \\
            & & 0 & -1 \\
            & & 1 & 0
        \end{bmatrix} = \begin{bmatrix}
            & -1 & & \\
            c & & & s \\
            s & & & -c \\
            & & -1 &
        \end{bmatrix}, 
    \end{equation}
    meaning 
    \begin{equation}
        I_0Q + QI_0 = \begin{bmatrix}
            & -1-c & s & \\
            1+c & & & s \\
            s & & & 1-c \\
            & s & -1+c &
        \end{bmatrix}.
    \end{equation}
    It follows that 
    \begin{align}
        \det(I_0Q + QI_0) &= \det\begin{bmatrix}
            1+c & s \\
            s & 1-c
        \end{bmatrix} \cdot \det \begin{bmatrix}
            -1-c & s \\
            s & -1+c
        \end{bmatrix} \\
        &= (1-c^2 - s^2)(1-c^2-s^2) \\
        &= (1-1)(1-1) \\
        &= 0,
    \end{align}
    and we are done.
\end{proof}

\begin{theorem}\thlabel{oppositeCircles}
    Let $J$ and $K$ be orthogonal complex structures of opposite sign on $\R^{2n}$. Then $J + K$ has nontrivial kernel.
\end{theorem}
\begin{proof}
    We proceed by induction on $n$. The only orthogonal complex structures on $\R^2$ are $\begin{bmatrix}
        0 & -1 \\ 1 & 0
    \end{bmatrix}$, which is positive, and $\begin{bmatrix}
        0 & 1 \\ -1 & 0
    \end{bmatrix}$, which is negative. The sum of these two matrices has nontrivial kernel, and so the claim holds for $n=1$. By \thref{base_case}, the claim holds for $n = 2$.
    
    Now assume $n > 2$. Let the function $\varphi: S^{2n-1} \to [-1,1]$ be given by $\varphi(v) = \langle Jv, Kv\rangle$. If we can show that $-1 \in \operatorname{im}(\varphi)$ then we're done, since 
    \begin{align}
        \varphi(v) = -1 &\implies \langle Jv, Kv\rangle = -1 \\
        &\implies Jv = -Kv \\
        &\implies (J+K)(v) = 0.
    \end{align}
    
    \underline{Case 1:} Suppose we are lucky enough to find that $\varphi$ has no critical values in the interval $(-1,1)$. Since $S^{2n-1}$ is connected and compact, $\operatorname{im}(\varphi)$ must equal some closed interval $[a,b]$ with $-1 \leq a \leq b \leq 1$. Since $a = \min(\varphi)$ and $b = \max(\varphi)$ are critical values of $\varphi$, by our assumption we have $a,b \notin (-1,1)$. So the only possibilities for $\im(\varphi)$ are $[-1, 1]$, $\{-1\}$, or $\{1\}$. But $\operatorname{im}(\varphi)$ cannot equal $\{1\}$ or else $J = K$, contradicting the fact that they have opposite signs. Therefore, $-1 \in \operatorname{im}(\varphi)$, and we would be done.

    \underline{Case 2:} Suppose that $\varphi$ has some critical point $p \in S^{2n-1}$ with $\varphi(p) \in (-1,1)$. The fact that $p$ is a critical point tells us that for any vector $v$ orthogonal to $p$ we have
    \begin{equation}
        \frac{d}{dt}\Big|_{t = 0} \varphi(p + tv) = 0.
    \end{equation}
    After expanding, the left side reduces to $\langle Jp, Kv\rangle + \langle Jv, Kp\rangle$. So 
    \begin{align}
        0 &= \langle Jp, Kv\rangle + \langle Jv, Kp\rangle \\
        &= \langle KJp, K^2v\rangle + \langle J^2v, JKp\rangle \qquad \text{since $J$ and $K$ are orthogonal} \\
        &= \langle KJp, -v\rangle + \langle -v, JKp\rangle \\
        &= -\langle v, JKp + KJp\rangle.
    \end{align}
    Letting $w = JKp + KJp$, we have just shown that $v \perp p \implies v\perp w$ for all $v \in \R^{2n}$. Therefore, $w \in \lspan(p)$.
    Now we apply \thref{basisChoice} to $J, K, p$ to get $\mathcal{E}$, $\mathcal{F}$, $Q$ satisfying the conclusions of that lemma. We know that $w \in \lspan(p) = \lspan(e_1)$. To make sense of this, we want to express $w$ in terms of the basis $\mathcal{E}$.
    \begin{align}
        w &= JKp + KJp \\
        &= JKf_1 + KJe_1 \\
        &= Jf_2 + Ke_2.
    \end{align}
    Using the change of basis matrix $Q$, we express $J$, $K$, and $f_2$ as matrices with respect to the basis $\mathcal{E}$.
    Letting $I_0$ be the standard complex structure, we have
    \begin{align}
        J &= I_0, \\
        K &= QI_0Q^T, \\
        f_2 &= Q(e_2).
    \end{align}
    Thus, 
    \begin{equation}
        w = I_0 Q(e_2) + QI_0Q^T(e_2),
    \end{equation}
    which may be directly computed to equal $-2c_1 e_1 + s_1(1-c_2) e_4 - s_1s_2e_5$.
    Since $w \in \lspan(e_1)$, we have that $s_1(1-c_2) = s_1s_2 = 0$. 
    
    If $s_1 = 0$, then the fact that $c_1^2 + s_1^2 = 1$ tells us that $c_1 = \pm1$, meaning $e_2 = \pm f_2$, contradicting the fact that $\varphi(p) = \langle e_2, f_2\rangle \in (-1,1)$. So we may assume that $s_1 \neq 0$. Then we must have $1 - c_2 = s_2 = 0$, and it follows that $Q(e_4) = e_4$. Therefore, $\lspan\{e_1, \cdots, e_4\}$ is a $Q$-invariant subspace. But $Q(e_i) = f_i$ for all $i$, so this means 
    \begin{equation}
        \lspan\{e_1, e_2, e_3, e_4\} = \lspan\{Qe_1, Qe_2, Qe_3, Qe_4 \} = \lspan\{f_1, f_2, f_3, f_4\}.
    \end{equation}
    Call the above subspace $W$.
    We find that $W$ is $J$-invariant (since $J$ fixes $\lspan\{e_1, \cdots, e_4\}$) and $K$-invariant (since $K$ fixes $\lspan\{f_1, \cdots, f_4\}$). 
    The map $Q|_W$ has matrix representation equal to the top left $4 \times 4$ submatrix of $Q$, and this submatrix has positive determinant due to the fact that $c_2 = 1$. Since $\det(Q|_W)$ is positive and $\det(Q)$ is negative, we must have that $\det(Q|_{W^\perp})$ is negative. Thus, $J|_{W^\perp}$ and $K|_{W^\perp}$ are orthogonal complex structures on $W^\perp \cong \R^{2n - 4}$ of opposite sign. Applying our inductive hypothesis, we find that their sum $(J + K)|_{W^\perp}$ has nontrivial kernel. Thus, $J+K$ has nontrivial kernel, and we are done.
\end{proof}

\begin{corollary}\thlabel{main_result_ocs}
    Let $J$ and $K$ be orthogonal complex structures on $\R^{2n}$, and suppose one of the following holds: 
    \begin{enumerate}
        \item $J$ and $K$ have opposite sign and $n$ is even.
        \item $J$ and $K$ have the same sign and $n$ is odd.
    \end{enumerate}
    Then $\ker(J-K)$ is nontrivial.
\end{corollary}
\begin{proof}
    By \thref{signOfNegative}, $-K$ has the same sign as $K$ when $n$ is even, and $-K$ has sign opposite that of $K$ when $n$ is odd. The claim follows from applying \thref{oppositeCircles} to $J$, $-K$.
\end{proof}
In light of \thref{win_condition}, the above corollary is equivalent to \thref{main_result}, and so we have proven our main result. We now use \thref{oppositeCircles} to prove \thref{main_result_corollary}, which is slightly stronger but whose interpretation in terms of Hopf fibrations is less straightforward.

\begin{lemma}\thlabel{kernelInvariant}
    Let $J$ and $K$ be complex structures on $\R^{2n}$. Then the subspace $\ker(J+K)$ is $J$-invariant. Additionally, any $J$-invariant subspace of $\ker(J+K)$ is also $K$-invariant.
\end{lemma}
\begin{proof}
    We have that 
    \begin{align}
        v \in \ker(J+K) &\implies Jv + Kv = 0 \\
        &\implies KJv + K^2v = 0 \\
        &\implies KJv + J^2v = 0 \qquad \text{since $K^2 = -\id = J^2$} \\
        &\implies Jv \in \ker(J+K).
    \end{align}
    Therefore, $\ker(J+K)$ is $J$-invariant. Now let $W$ be a $J$-invariant subspace of $\ker(J+K)$. 
    \begin{align}
        W \subseteq \ker(J+K) &\implies (J+K)|_W = 0 \\ &\implies K|_W = -J|_W.
    \end{align}
    Therefore, $K(W) = -J(W) = W$, meaning $W$ is also $K$-invariant.
\end{proof}

\begin{corollary}\thlabel{main_result_corollary}
    Let $J$ and $K$ be orthogonal complex structures on $\R^{2n}$. If $J$ and $K$ have opposite sign, then $\dim(\ker(J + K)) \equiv 2 \mod 4$. If $J$ and $K$ have the same sign, then $\dim(\ker(J+K)) \equiv 0 \mod 4$.
\end{corollary}
\begin{proof}
    We work by induction on $n$. The claim is easily checked for $n = 1$, so assume $n > 1$. 
    
    First, suppose $J$ and $K$ have opposite sign. By \thref{oppositeCircles}, $\ker(J+K)$ is nontrivial, so let $v \in \ker(J+K)$ be a nonzero vector, and let $W = \lspan\{v, Jv\}$. By \thref{kernelInvariant}, $W$ is a 2-dimensional subspace of $\ker(J+K)$ which is both $J$-invariant and $K$-invariant. Since $K|_W = -J|_W$, by \thref{signOfNegative} we have that $J|_W$ and $K|_W$ have opposite signs. Therefore, $J|_{W^\perp}$ and $K|_{W^\perp}$ have the same sign. By our inductive hypothesis, $\dim(\ker((J+K)|_{W^\perp})) \equiv 0 \mod 4$. 
    Adding back in the contribution from $W$, we get that $\dim(\ker(J+K)) \equiv 2 \mod 4$, as desired.

    Now suppose $J$ and $K$ have the same sign. If $\ker(J+K)$ is trivial then $\dim(\ker(J+K)) = 0$ and we're done, so assume it's nontrivial. As before, we may choose $W\subseteq \ker(J+K)$ to be a 2-dimensional subspace which is both $J$-invariant and $K$-invariant. $J|_W$ and $K|_W$ have opposite signs, so $J|_{W^\perp}$ and $K|_{W^\perp}$ have opposite signs. By our inductive hypothesis, $\dim(\ker((J+K)|_{W^\perp})) \equiv 2 \mod 4$. Adding back in the contribution from $W$, we get that $\dim(\ker(J+K)) \equiv 0 \mod 4$, as desired.
\end{proof}

\section{Next steps}\label{next_steps}\label{S5byS1}

Here is a brief summary of what we have accomplished in this paper: 
Let $X$ denote the space of distance-decreasing maps $S^2 \to S^2$, and let $\mathcal{F}^+$ denote the space of positive-sign fibrations of $S^3$ by oriented great circles. Gluck and Warner constructed a map $X \to \mathcal{F}^+$ which maps a function $f: S^2 \to S^2$ to the graph of $f$ in $S^2_- \times S^2_+ \cong G_2^+(\R^4)$, and as we saw in \thref{distance_decreasing}, they proved this map is a well-defined bijection. In \thref{new_formulation}, we reformulated Gluck and Warner's map in terms of circles of agreement of Hopf fibrations, and in \thref{main_result_ocs} we showed that these circles of agreement are guaranteed to exist for higher-dimensional Hopf fibrations as well.

In light of what we've done, the natural next step would be to state and prove a version of \thref{new_formulation} for higher-dimensional Hopf fibrations, but we run into an obstacle: Although circles of agreement are guaranteed to exist for higher-dimensional Hopf fibrations, they are no longer guaranteed to be unique. Here's a simple example: 
\begin{example}
Let $J$ and $K$ be orthogonal complex structures on $\R^{2n}$, and let $I = \begin{bmatrix}
        0 & -1 \\ 1 & 0
    \end{bmatrix}$ denote the standard complex structure on $\R^2$. The following chart shows that no matter what restrictions we place on the parity of $n$ or the signs of $J$ and $K$, it is always possible for $\ker(J-K)$ to have dimension greater than 2:
\begin{center}
\begin{tabular}{|Sc||Sc|Sc|}
\hline
 & $J, K$ same sign & $J, K$ opposite signs \\
\hline
\hline
$n$ odd &
$J = K = \begin{bmatrix} I & 0 & 0 \\ 0 & I & 0 \\ 0 & 0 & I \end{bmatrix}$ &
$J = \begin{bmatrix} I & 0 & 0 \\ 0 & I & 0 \\ 0 & 0 & I \end{bmatrix}, \quad
K = \begin{bmatrix} -I & 0 & 0 \\ 0 & I & 0 \\ 0 & 0 & I \end{bmatrix}$ \\
\hline
$n$ even &
$J = K = \begin{bmatrix} I & 0 \\ 0 & I \end{bmatrix}$ &
$J = \begin{bmatrix} I & 0 & 0 & 0 \\ 0 & I & 0 & 0 \\ 0 & 0 & I & 0 \\ 0 & 0 & 0 & I \end{bmatrix}, \quad
K = \begin{bmatrix} -I & 0 & 0 & 0 \\ 0 & I & 0 & 0 \\ 0 & 0 & I & 0 \\ 0 & 0 & 0 & I \end{bmatrix}$ \\
\hline
\end{tabular}
\end{center}
\end{example}
Without uniqueness, it's not entirely clear how to generalize the map $\varphi_p: S^2 \times S^2 \to G_2^+(\R^4)$ given in \thref{hopf_homeo} to higher dimensions, but there are some things we could try. If we focus our attention on fibrations of $S^5$, we do have a weaker form of uniqueness: By \thref{main_result_corollary}, if $H_1$ and $H_2$ are Hopf fibrations of $S^5$ with the same sign, then $\dim(\ker(I_{H_1} - I_{H_2})) \in \{2, 6\}$. In other words, $H_1$ and $H_2$ either share a unique oriented great circle, or are equal. 

So one possible way we might generalize \thref{hopf_homeo} is as follows: Let $\mathcal{H}^+$ denote the space of positive Hopf fibrations of $S^5$.
\begin{definition}[The map $\varphi: \{(H_1, H_2) \in \mathcal{H}^+ \times \mathcal{H}^+ \mid H_1 \neq H_2\} \to G_2^+(\R^6)$]
    Given positive Hopf fibrations $H_1, H_2 \in \mathcal{H}^+$ with $H_1 \neq H_2$, we define $\varphi(H_1, H_2)$ to equal the oriented 2-plane containing the unique circle of agreement of $H_1$ and $H_2$.
\end{definition}

One potential issue with this approach is that $\varphi$ is no longer injective. There are an infinite number (in fact, an $S^2$'s worth) of positive Hopf fibrations of $S^5$ containing any given oriented great circle, so there are an infinite number of choices of $(H_1, H_2)$ which $\varphi$ map to the same oriented 2-plane in $G_2^+(\R^6)$. This could be solved by taking an appropriate quotient of the domain. Perhaps future work could examine if this direction bears any fruit.

\section{Fibrations of $S^7$ by $S^3$}\label{S3_fibrations}

Gluck and Warner classified the fibrations of $S^n$ by great $k$-spheres in the case of $n = 3$, $k = 1$. So far, all of our attempts to generalize their result have aimed to increase the value of $n$. This final section is devoted to examining what happens when we try to increase $k$. 

We will first define Hopf fibrations of $S^7$ by $S^3$, which are the simplest nontrivial examples of fibrations of $S^n$ by great $k$-spheres where $k > 1$. We might hope that an analogue to \thref{opp_sign_share} holds. That is, any two opposite-sign Hopf fibrations of $S^7$ by $S^3$ are guaranteed to share a unique oriented great 3-sphere. However, this turns out to be false. In contrast to Section~\ref{S5byS1} where we observed that we had existence but not uniqueness, we will see in Theorems \ref{shared_S3_uniqueness} and \ref{S3_counterexample} that for Hopf fibrations of $S^7$ by $S^3$ of opposite sign, uniqueness of a shared $S^3$ holds, and it's actually existence that fails. 

Recall that the quaternions are a 4-dimensional non-commutative normed division algebra over $\R$ given by letting $(1,i,j,k)$ be an orthonormal basis for $\R^4$ and defining multiplication by $i^2 = j^2 = k^2 = ijk = -1$. We write $\H$ to denote the set of quaternions.
\begin{definition}[Standard Hopf fibration of $S^{4n-1}$ by $S^3$]
    View $S^{4n-1}$ as the unit sphere in $\R^{4n} \cong \H^n$, and view $S^3 \subseteq \H$ as the set of unit quaternions. $S^3$ acts on $S^{4n-1}$ via left-multiplication. We define the standard Hopf fibration to be the fibration of $S^{4n-1}$ whose fibers are the orbits of this action. Note that the fibers are oriented great $3$-spheres of $S^{4n-1}$.
\end{definition}
\begin{definition}[Hopf fibration of $S^{4n-1}$ by $S^3$]\thlabel{S7_hopf_sign}
    View $S^{4n-1}$ as the unit sphere in $\R^{4n}$, and let $H$ be the standard Hopf fibration of $S^{4n-1}$ by $S^3$. We define a Hopf fibration to be a fibration of the form $T(H)$, where $T: \R^{4n} \to \R^{4n}$ is an orthogonal linear transformation. If $\det(T) > 0$ then we say the Hopf fibration has positive sign. If $\det(T) < 0$ then we say it has negative sign. 
\end{definition}

\subsection{Hopf fibrations of $S^7$ by $S^3$ are linear}

This section closely mirrors the structure of Section~\ref{hopf_linear}. Our goal is to establish a correspondence $H$ between the soon-to-be-defined set of \emph{orthogonal quaternionic structures} on $\R^8$ and the set of Hopf fibrations of $S^7$ by $S^3$. Unlike in Section~\ref{hopf_linear}, this correspondence is not bijective, but we will see in \thref{correspondence2} that it is at least surjective. Because the arguments in this section are similar to those in Section~\ref{hopf_linear}, we will omit many of the proofs but include a reference to the corresponding result from Section~\ref{hopf_linear} when appropriate. 

\begin{definition}[Quaternionic structure, \ref{complex_structure}, \ref{orthogonal_complex_structure}]\thlabel{ortho_quat_struct}
    We define a quaternionic structure on $\R^{4n}$ to be a 3-tuple $(I, J, K)$ of complex structures on $\R^{4n}$ which satisfy $IJK = -\id$. If in addition $I$, $J$, and $K$ are orthogonal, then we call $(I,J,K)$ an orthogonal quaternionic structure.
\end{definition}

\begin{proposition}
    Let $(I,J,K)$ be an orthogonal quaternionic structure. Then the vectors $p$, $Ip$, $Jp$, and $Kp$ are mutually orthogonal.
\end{proposition}
\begin{proof}
    $p$ is orthogonal to $Ip$, $Jp$, and $Kp$ since $I$, $J$, and $K$ are skew-symmetric. $Ip$ is orthogonal to $Jp$ because 
    \begin{align}
        \langle Ip, Jp \rangle &= \langle I^2p, IJp \rangle \quad \text{since $I$ is orthogonal} \\
        &= \langle -p, Kp \rangle \\
        &= 0 \hspace{1.99 cm} \text{since $K$ is skew-symmetric}.
    \end{align}
    The same argument can be used to show $\langle Jp, Kp \rangle = 0$ and $\langle Kp, Ip \rangle = 0$.
\end{proof}

\begin{definition}[Sign of a quaternionic structure, \ref{signCheck}]\thlabel{sign_of_quat}
    Let $(I,J,K)$ be a quaternionic structure on $\R^{4n}$. Fix a decomposition $\R^{4n} = W_1 \oplus \cdots \oplus W_n$ where each $W_k$ is a 4-dimensional subspace invariant under $I$, $J$ and $K$. For each $k \in \{1, \cdots, n\}$, choose a nonzero vector $e_k \in W_k$. We define the sign of $(I,J,K)$ to equal the sign of the determinant of the $4n \times 4n$ matrix
    \begin{equation}\label{quat_sign_matrix}
        \begin{bmatrix}
            e_1 & Ie_1 & Je_1 & Ke_1 & \cdots & e_n & Ie_n & Je_n & Ke_n
        \end{bmatrix},
    \end{equation}
    where each entry is viewed as a column vector. 
\end{definition}

\begin{proposition}\thlabel{sign_quat_to_comp}
    Let $(I,J,K)$ be a quaternionic structure on $\R^{4n}$. Then the sign of $(I,J,K)$ as a quaternionic structure must agree with the signs of $I$, $J$, and $K$ as complex structures.
\end{proposition}
\begin{proof}
    We may rewrite the matrix \eqref{quat_sign_matrix} as 
    \begin{equation}
        \begin{bmatrix}
            e_1 & I(e_1) & Je_1 & I(Je_1) & \cdots & e_n & I(e_n) & Je_n & I(Je_n)
        \end{bmatrix}.
    \end{equation}
    By \thref{signCheck}, the sign of the determinant of this matrix equals the sign of $I$ as a complex structure, and thus $(I,J,K)$ and $I$ have the same sign. $J$ and $K$ follow similarly.
\end{proof}

Let $\mathcal{Q}$ denote the set of quaternionic structures on $\R^{4n}$ and let $\mathcal{Q}_\perp \subseteq \mathcal{Q}$ denote the set of orthogonal quaternionic structures. Let $\mathcal{F}$ denote the set of fibrations of $S^{4n-1} \subseteq \R^{4n}$ by oriented great 3-spheres.
\begin{definition}[The map $H: \mathcal{Q} \to \mathcal{F}$, \ref{the_map_I}]\thlabel{the_map_H}
    Given a quaternionic structure $(I,J,K)$ on $\R^{4n}$, we define $H_{(I,J,K)}$ to be the fibration of $S^{4n-1}$ whose fibers are of the form $S^{4n-1} \cap \lspan(p, Ip, Jp, Kp)$ as $p$ ranges over $S^{4n-1}$. Note that this is well-defined since the 4-plane $\lspan(p, Ip, Jp, Kp)$ is simultaneously $I$-, $J$-, and $K$-invariant.
\end{definition}

\begin{proposition}[\ref{win_condition}]\thlabel{S7_agreement_criterion}
    Let $(I_1, J_1, K_1)$ and $(I_2, J_2, K_2)$ be two quaternionic structures on $\R^{4n}$, and let $p \in S^{4n-1}$ be a unit vector. Then $H_{(I_1, J_1, K_1)}$ and $H_{(I_2, J_2, K_2)}$ share an oriented fiber containing $p$ if and only if 
    \begin{equation}
    \lspan(p, I_1p, J_1p, K_1p) = \lspan(p, I_2p, J_2p, K_2p).
    \end{equation}
\end{proposition}
\begin{proof}
    This directly follows from \thref{the_map_H}.
\end{proof}

\begin{definition}[Standard quaternionic structure, \ref{standard_complex_structure}]
Identify the quaternions $\H$ with $\R^4$, and let $L_i, L_j, L_k: \R^4 \to \R^4$ to be the matrices corresponding to left-multiplication by $i$, $j$, and $k$, respectively: 
    \begin{equation}
        L_i = \begin{bmatrix}
            0 & -1 & 0 & 0 \\
            1 & 0 & 0 & 0 \\
            0 & 0 & 0 & -1 \\
            0 & 0 & 1 & 0
        \end{bmatrix}, \quad
        L_j = \begin{bmatrix}
            0 & 0 & -1 & 0 \\
            0 & 0 & 0 & 1 \\
            1 & 0 & 0 & 0\\
            0 & -1 & 0 & 0
        \end{bmatrix}, \quad 
        L_k = \begin{bmatrix}
            0 & 0 & 0 & -1 \\
            0 & 0 & -1 & 0 \\
            0 & 1 & 0 & 0 \\
            1 & 0 & 0 & 0
        \end{bmatrix}.
    \end{equation}
    We define the standard quaternionic structure on $\R^{4n}$ to equal 
    \begin{equation}
        \mleft( \begin{bmatrix}
            L_i & \\
            & \ddots \\
             & & L_i
        \end{bmatrix}, \begin{bmatrix}
            L_j &  \\
            & \ddots \\
             & & L_j
        \end{bmatrix}, \begin{bmatrix}
            L_k &  \\
            & \ddots \\
             & & L_k
        \end{bmatrix}\mright).
    \end{equation}
    It's easy to see that this is mapped by $H$ to the standard Hopf fibration.
\end{definition}

\begin{fact}[\ref{ocs_conjugation}, \ref{conjugation}]\thlabel{quat_conjugation}
    Let $(I,J,K)$ be an orthogonal quaternionic structure on $\R^{4n}$ and let $T: \R^{4n} \to \R^{4n}$ be an orthogonal transformation. Then $(TIT^{-1}, TJT^{-1}, TKT^{-1})$ is an orthogonal quaternionic structure. Furthermore, $T(H_{(I, J, K)}) = H_{(TIT^{-1}, TJT^{-1}, TKT^{-1})}$.
\end{fact}


\begin{fact}[\ref{correspondence}]\thlabel{correspondence2}
    The map $H|_{\mathcal{Q}_\perp}$ is a surjection from the set of orthogonal quaternionic structures onto the set of Hopf fibrations of $S^{4n-1}$ by $S^3$.
\end{fact}

\subsection{Uniqueness of a shared 3-sphere}
\begin{theorem}\thlabel{shared_S3_uniqueness}
    Let $H^+$ and $H^-$ be Hopf fibrations of $S^7$ by $S^3$ of opposite sign. Then $H^+$ and $H^-$ share at most one oriented great 3-sphere.
\end{theorem}
\begin{proof}
    By \thref{correspondence2}, we may choose a pair of orthogonal quaternionic structures $(I^+, J^+, K^+)$ and $(I^-, J^-, K^-)$ on $\R^8$ such that $H_{(I^+, J^+, K^+)} = H^+$ and $H_{(I^-, J^-, K^-)} = H^-$. 
    
    Now, suppose for contradiction that $H^+$ and $H^-$ share two distinct oriented 3-spheres. Let $P,Q \in G_4^+(\R^8)$ be the oriented planes containing these 3-spheres, and let $p$ and $q$ be unit vectors in $P$ and $Q$, respectively. By \thref{S7_agreement_criterion} we have
    \begin{align}
        P &= \lspan(p, I^+p, J^+p, K^+p) = \lspan(p, I^-p, J^-p, K^-p), \label{first1} \\
        Q &= \lspan(q, I^+q, J^+q, K^+q) = \lspan(q, I^-q, J^-q, K^-q).\label{second1}
    \end{align}
    By \thref{sign_of_quat} we have that the sign of $(I^+, J^+, K^+)$ equals the sign of the determinant of
    \begin{equation}
        \begin{bmatrix}
            p & I^+p & J^+p & K^+p & q & I^+q & J^+q & K^+q
        \end{bmatrix},
    \end{equation}
    which by \eqref{first1} and \eqref{second1} equals the sign of the determinant of 
    \begin{equation}
        \begin{bmatrix}
            p & I^-p & J^-p & K^-p & q & I^-q & J^-q & K^-q
        \end{bmatrix},
    \end{equation}
    which by \thref{sign_of_quat} equals the sign of $(I^-, J^-, K^-)$. Thus, $(I^+, J^+, K^+)$ and $(I^-, J^-, K^-)$ have the same sign, meaning $H^+$ and $H^-$ have the same sign, a contradiction.
\end{proof}

The above theorem shows that given a Hopf fibration of $S^7$ by $S^3$, knowing its orientation on a pair of orthogonal 3-spheres completely determines its sign. However, it's worthwhile to note that unlike in the case of great circle fibrations, this is actually not enough to determine the fibration itself:
\begin{theorem}
    Identify $\R^8$ with $\H^2$, let 
    \begin{equation}
        (I_1, J_1, K_1) = \mleft( \begin{bmatrix}
            L_i & 0 \\
            0 & L_i
        \end{bmatrix}, \begin{bmatrix}
            L_j & 0 \\ 0 & L_j
        \end{bmatrix}, \begin{bmatrix}
            L_k & 0 \\ 0 & L_k
        \end{bmatrix}\mright)
    \end{equation}
    be the standard quaternionic structure, and let 
    \begin{equation}
        (I_2, J_2, K_2) = \mleft( \begin{bmatrix}
            L_i & 0 \\
            0 & L_i
        \end{bmatrix}, \begin{bmatrix}
            L_j & 0 \\ 0 & -L_j
        \end{bmatrix}, \begin{bmatrix}
            L_k & 0 \\ 0 & -L_k
        \end{bmatrix}\mright)
    \end{equation}
    be the standard quaternionic structure conjugated by $\begin{bmatrix}
        \textnormal{id} & 0 \\ 0 & L_i
    \end{bmatrix}$. 
    
    \smallskip
    
    \noindent Then the Hopf fibrations $H_{(I_1, J_1, K_1)}$ and $H_{(I_2, J_2, K_2)}$ share a pair of orthogonal 3-spheres and yet are not equal.
\end{theorem}
\begin{proof}
    By \thref{S7_agreement_criterion}, to show that $H_{(I_1, J_1, K_1)}$ and $H_{(I_2, J_2, K_2)}$ share a pair of orthogonal 3-spheres it suffices to find a pair of vectors $p,q \in S^7$ such that 
    \begin{equation}
        \underbrace{\lspan(p, I_1p, J_1p, K_1p)}_{\text{call this } P_1} = \underbrace{\lspan(p, I_2p, J_2p, K_2p)}_{\text{call this }P_2}
    \end{equation}
    and 
    \begin{equation}
        \underbrace{\lspan(q, I_1q, J_1q, K_1q)}_{\text{call this }Q_1} = \underbrace{\lspan(q, I_2q, J_2q, K_2q)}_{\text{call this }Q_2}
    \end{equation}
    with $P_1$ and $Q_1$ orthogonal. Letting $p = \begin{bmatrix}
        1 \\ 0
    \end{bmatrix}$ and $q = \begin{bmatrix}
        0 \\ 1
    \end{bmatrix}$, we get 
    \begin{equation}
        P_1 = P_2 = \lspan_\R\mleft( \begin{bmatrix}
        1 \\ 0
    \end{bmatrix}, \begin{bmatrix}
        i \\ 0
    \end{bmatrix}, \begin{bmatrix}
        j \\ 0
    \end{bmatrix}, \begin{bmatrix}
        k \\ 0
    \end{bmatrix} \mright)
    \end{equation}
    and 
    \begin{equation}
    Q_1 = Q_2 = \lspan_\R\mleft( \begin{bmatrix}
        0 \\ 1
    \end{bmatrix}, \begin{bmatrix}
        0 \\ i
    \end{bmatrix}, \begin{bmatrix}
        0 \\ j
    \end{bmatrix}, \begin{bmatrix}
        0 \\ k
    \end{bmatrix} \mright),
    \end{equation}
    and these are indeed orthogonal.

    To show that $H_{(I_1, J_1, K_1)}$ and $H_{(I_2, J_2, K_2)}$ are not the same fibration, it suffices to find a vector $r$ such that 
    \begin{equation}
        \underbrace{\lspan(r, I_1r, J_1r, K_1r)}_{\text{call this }R_1} \neq \underbrace{\lspan(r, I_2r, J_2r, K_2r)}_{\text{call this }R_2}.
    \end{equation}
    Letting $r = \begin{bmatrix}
        1 \\ 1
    \end{bmatrix}$, we get 
    \begin{equation}
        R_1 = \lspan_\R\mleft( 
        \begin{bmatrix}
            1 \\ 1
        \end{bmatrix}, 
        \begin{bmatrix}
            i \\ i
        \end{bmatrix},
        \begin{bmatrix}
            j \\ j
        \end{bmatrix},
        \begin{bmatrix}
            k \\ k
        \end{bmatrix}
        \mright) \neq \lspan_\R\mleft( 
        \begin{bmatrix}
            1 \\ 1
        \end{bmatrix}, 
        \begin{bmatrix}
            i \\ i
        \end{bmatrix},
        \begin{bmatrix}
            j \\ -j
        \end{bmatrix},
        \begin{bmatrix}
            k \\ -k
        \end{bmatrix}
        \mright) = R_2,
    \end{equation}
    completing the proof.
\end{proof}
Thus, unlike the situation with great circle fibrations where the set of points on which two Hopf fibrations agree is the kernel of a linear map (\thref{win_condition}), we find that the set of points on which two Hopf fibrations by $S^3$ agree need not even be a linear subspace.

\subsection{Failure of existence of a shared 3-sphere}
We first prove a lemma that will allow us to detect when two orthogonal quaternionic structures share a 4-dimensional invariant subspace:
\begin{lemma}\thlabel{detector}
    Let $(I^+, J^+, K^+)$ and $(I^-, J^-, K^-)$ be orthogonal quaternionic structures of opposite sign on $\R^4$. Then 
    \begin{equation}
        \ker(I^- + I^+) \cap \ker(J^- + J^+) \cap \ker(K^- + K^+) \neq \{0\}.
    \end{equation}
\end{lemma}
\begin{proof}
    By \thref{sign_quat_to_comp}, $I^-$ and $I^+$ are orthogonal complex structures of opposite sign on $\R^4$. So by \thref{oppositeCircles} there exists some $p \in S^3$ for which $I^-(p) + I^+(p) = 0$. 
    
    Let $\mathcal{E} = (p, I^+p, J^+p, K^+p)$ and $\mathcal{F} = (p, I^-p, J^-p, K^-p)$ be ordered bases for $\R^4$, and let $Q$ be the change of basis matrix from $\mathcal{E}$ to $\mathcal{F}$. Since $p = p$ and $I^-p = -I^+p$, we find that $Q$ is of the form 
    \begin{equation}
        \begin{bmatrix}
            1 & 0 & a & b\\
            0 & -1 & c & d \\
            0 & 0 & e & f \\
            0 & 0 & g & h
        \end{bmatrix}.
    \end{equation}
    Since $\mathcal{E}$ and $\mathcal{F}$ are both orthonormal but with opposite orientation, $Q$ must be orthogonal with negative determinant, from which we deduce that 
    \begin{equation}
        a = b = c = d = 0 \qquad \text{and} \qquad \begin{bmatrix}
        e & f \\ g & h
    \end{bmatrix} \in \operatorname{SO}(2).
    \end{equation}
    It follows that 
    \begin{equation}
        Q = \begin{bmatrix}
            1 & 0 & \\
            0 & -1 &  \\
             &  & c & -s \\
             &  & s & c
        \end{bmatrix}
    \end{equation}
    for some real numbers $c$ and $s$ satisfying $c^2 + s^2 = 1$.

    Our choice of $\mathcal{E}$ and $\mathcal{F}$ makes it so that
    \begin{enumerate}
        \item $(I^+, J^+, K^+)$ is the standard quaternionic structure in the basis $\mathcal{E}$.
        \item $(I^-, J^-, K^-)$ is the standard quaternionic structure in the basis $\mathcal{F}$.
    \end{enumerate}
    So writing everything in terms of the basis $\mathcal{E}$, we have 
    \begin{align}
        I^+ &= L_i, & J^+ &= L_j, & K^+ &= L_k, \\
        I^- &= QL_iQ^{-1}, & J^- &= QL_jQ^{-1}, & K^- &= QL_kQ^{-1}.
    \end{align}
    Our goal is now to show 
    \begin{equation}
        \ker(QL_iQ^{-1} + L_i) \cap \ker(QL_jQ^{-1} + L_j) \cap \ker(QL_kQ^{-1} + L_k) \neq \{0\},
    \end{equation}
    which is equivalent to 
    \begin{equation}
        \ker(QL_i + L_iQ) \cap \ker(QL_j + L_jQ) \cap \ker(QL_k + L_kQ) \neq \{0\}.
    \end{equation}
    We compute that

    \begin{align}
        \ker(QL_i + L_iQ) &= \ker\begin{bmatrix}
            0 & 0 \\
            0 & 0 \\
            & & -s & -c \\
            & & c & -s
        \end{bmatrix}, \\
        \ker(QL_j + L_jQ) &= \ker\begin{bmatrix}
            & & -1-c & s \\
            & & s & -1 + c \\
            1 + c & s \\
            s & 1-c
        \end{bmatrix}, \\
    \ker(QL_k + L_kQ) &= \ker\begin{bmatrix}
            & & -s & -1-c \\
            & & 1-c & s \\
            -s & -1+c \\
            1+c & s
        \end{bmatrix}.
    \end{align}
    It's easy to check via matrix multiplication that the vectors 
    \begin{equation}
        \begin{bmatrix}
        -s \\ 1+c \\ 0 \\ 0
    \end{bmatrix} \qquad \text{and} \qquad \begin{bmatrix}
        1-c \\ -s \\ 0 \\ 0
    \end{bmatrix}
    \end{equation}
    both lie in the intersection of all three kernels. No matter the choice of $c$ and $s$, at least one of these vectors is nonzero, and so the intersection of the kernels is nontrivial, as desired.
\end{proof}

\begin{theorem}\thlabel{S3_counterexample}
    Let $(I^+, J^+, K^+)$ be the standard quaternionic structure on $\R^8$ and let $(I^-, J^-, K^-) = (QI^+Q^{-1}, QJ^+Q^{-1}, QK^+Q^{-1})$, where 
    \begin{equation}
        Q = \begin{bmatrix}
            -1 & & & & & & & \\
            & 1 & & & & & & \\
            & & 1 & & & & & \\
            & & & 0 & -1 & & & \\
            & & & 1 & 0 & & & \\
            & & & & & 1 & & \\
            & & & & & & 1 & \\
            & & & & & & & 1
        \end{bmatrix}.
    \end{equation}
    Then $H_{(I^+, J^+, K^+)}$ and $H_{(I^-, J^-, K^-)}$ have no oriented 3-spheres in common. 
\end{theorem}
\begin{proof}
    Suppose for contradiction that $H_{(I^+, J^+, K^+)}$ and $H_{(I^-, J^-, K^-)}$ did have an oriented 3-sphere in common. Let $P \in G_4^+(\R^8)$ be the oriented plane containing this 3-sphere, and let $p$ be a unit vector in $P$. By \thref{S7_agreement_criterion} we have
    \begin{equation}
        P = \lspan(p, I^+p, J^+p, K^+p) = \lspan(p, I^-p, J^-p, K^-p).
    \end{equation}
    In particular, $P$ is simultaneously $I^+$-, $J^+$-, $K^+$-, $I^-$-, $J^-$-, and $K^-$-invariant. Since all of these are orthogonal maps, we have that $P^\perp$ is also invariant under the action of these six maps. So letting $q$ be a unit vector in $P^\perp$, we have that
    \begin{equation}
        P^\perp = \lspan\{q, I^+q, J^+q, K^+q\} = \lspan\{q, I^-q, J^-q, K^-q\}
    \end{equation}
    as unoriented planes. However, it cannot be the case that 
    \begin{equation}
        \lspan(q, I^+q, J^+q, K^+q) = \lspan(q, I^-q, J^-q, K^-q)
    \end{equation}
    as oriented planes or else $H_{(I^+, J^+, K^+)}$ and $H_{(I^-, J^-, K^-)}$ would share both $P$ and $P^\perp$, contradicting \thref{shared_S3_uniqueness}. It follows that the bases $(q, I^+q, J^+q, K^+q)$ and $(q, I^-q, J^-q, K^-q)$ have opposite orientation within $P^\perp$, and therefore $(I^+, J^+, K^+)|_{P^\perp}$ and $(I^-, J^-, K^-)|_{P^\perp}$ are orthogonal quaternionic structures on $P^\perp \cong \R^4$ of opposite sign. By \thref{detector}, we should have 
    \begin{equation}
        \ker(I^- + I^+) \cap \ker(J^- + J^+) \cap \ker(K^- + K^+) \neq \{0\},
    \end{equation}
    which is equivalent to 
    \begin{equation}
        \ker(QI^+ + I^+Q) \cap \ker(QJ^+ + J^+Q) \cap \ker(QK^+ + K^+Q) \neq \{0\}.
    \end{equation}
    But we can check that this is not true: 
    \begin{equation}
        QI^+ + I^+Q = \begin{bmatrix}
            0 & 0 \\
            0 & 0 \\
            & & 0 & -1 & 1 & 0 \\
            & & 1 & 0 & 0 & 1\\
            & & 1 & 0 & 0 & -1 \\
            & & 0 & 1 & 1 & 0 \\
            & & & & & & 0 & -2 \\
            & & & & & & 2 & 0
        \end{bmatrix},
    \end{equation}
    which has kernel $\lspan\{e_1, e_2\}$,
    \begin{equation}
        QJ^+ + J^+Q = \begin{bmatrix}
            & & & 0 & 0 \\
            & & & 1 & -1 \\
            & & & 0 & 0 \\
            0 & -1 & 0 &0&0&0&1&0 \\
            0&-1&0&0&0&0&-1&0 \\
            & & & 0 & 0 & & & 2 \\
            & & & 1 & 1 \\
            & & & 0 & 0 & -2
        \end{bmatrix},
    \end{equation}
    which has kernel $\lspan\{e_1, e_3\}$, and 
    \begin{equation}
        QK^+ + K^+Q = \begin{bmatrix}
            0 & 0 & 0 & 1 & 1\\
            0 & 0 & -2 & 0 \\
            0 & 2 & 0 & 0 \\
            -1 & 0 & 0 & 0 & & & & 1\\
            1 & & & & 0 & 0 & 0 & -1 \\
            & & & & 0 & 0 & -2 & 0 \\
            & & & & 0 & 2 & 0 & 0 \\
            & & & 1 & 1 & 0 & 0 & 0
        \end{bmatrix},
    \end{equation}
    which has kernel $\lspan\{ e_1 + e_8, \; e_4 - e_5 \}$.

    The intersection of these three kernels is $\{0\}$, and so our original assumption that $H_{(I^+, J^+, K^+)}$ and $H_{(I^-, J^-, K^-)}$ share an oriented 3-sphere must have been false.
\end{proof}

\appendix\section{Appendix}
Here we give proofs for some standard results used throughout the paper.
\begin{proposition}\thlabel{basis_of_exterior}
    Let $\mathcal{E} = \{e_1, e_2, e_3, e_4\}$ be an orthonormal basis for $\R^4$. Then the set 
    \begin{equation}
        \{e_i^*\wedge e_j^* \mid 1 \leq i < j \leq 4 \} = \{\omega_{\lspan(e_i,e_j)} \mid 1 \leq i < j \leq 4 \}
    \end{equation}
    is a basis for $\Lambda^2\R^4$. In particular, $\im(\omega)$ spans $\Lambda^2\R^4$.
\end{proposition}
\begin{proof}
    Our set is linearly independent because 
    \begin{align}
        \sum_{i < j} c_{i,j} [e_i^* \wedge e_j^*] = 0 &\implies \sum_{i < j} c_{i,j}  [e_i^* \wedge e_j^*](e_\ell, e_k) = 0 \quad \text{for all $\ell < k$} \\
        &\implies c_{\ell, k} = 0 \quad \text{for all $\ell<k$.}
    \end{align}
    Our set spans $\Lambda^2\R^4$ since for all $\alpha \in \Lambda^2\R^4$ we have
    \begin{equation}
        \alpha = \sum_{i < j} \alpha(e_i, e_j) [e_i^* \wedge e_j^*].
    \end{equation}
    To see why, observe that the left and right sides agree on all inputs in $\mathcal{E} \times \mathcal{E}$. By multilinearity, they agree on all inputs in $\R^4 \times \R^4$, and so they are equal.
\end{proof}
\subsection{Existence, uniqueness of the inner product on $\Lambda^2\R^4$}\label{induced_inner_product}
Here we prove that there is exactly one inner product on $\Lambda^2\R^4$ which satisfies the property given in \thref{inner_product}. Though we will only prove and use it for $\Lambda^2\R^4$, this easily generalizes to $\Lambda^k \R^n$.
\begin{lemma}[Cauchy-Binet formula]\thlabel{cauchy_binet}
    Let $A = [a_{i,j}]$ and $B = [b_{i,j}]$ be $4 \times 2$ matrices. Given $R \subseteq \{1,2,3,4\}$ and $S\subseteq\{1,2\}$, we write $A_{R\times S}$ to denote the submatrix of $A$ consisting of all $a_{i,j}$ for which $i \in R$ and $j \in S$. Then
    \begin{equation}
        \det(B^TA) = \sum_{k < \ell} \det(A_{\{k,\ell\} \times \{1,2\} }) \det(B_{\{ k,\ell \} \times \{1,2\}}).
    \end{equation}
\end{lemma}
\begin{proof}
    We compute that 
    \begin{equation}
        B^TA = \begin{bmatrix}
            \sum_k a_{k,1} b_{k,1} & \sum_k a_{k,2} b_{k,1} \\
            \sum_k a_{k,1} b_{k,2} & \sum_k a_{k,2} b_{k,2}
        \end{bmatrix}.
    \end{equation}
    Therefore, 
    \begin{align}
        \det(B^TA) &= \textstyle \Bigl( \sum_k a_{k,1} b_{k,1} \Bigr) \Bigl(\sum_\ell a_{\ell, 2}b_{\ell, 2} \Bigr) - \Bigl( \sum_{k} a_{k,1}b_{k,2} \Bigr) \Bigl( \sum_{\ell} a_{\ell,2}b_{\ell,1} \Bigr) \\
        &= \sum_{k, \ell} \underbrace{\big[ a_{k,1}a_{\ell,2}b_{k,1}b_{\ell,2} - a_{k,1} a_{\ell,2} b_{k,2}b_{\ell,1} \big]}_{\text{equals $0$ whenever $k = \ell$}} \\
        &= \sum_{k \neq \ell} \big[ a_{k,1}a_{\ell,2}b_{k,1}b_{\ell,2} - a_{k,1} a_{\ell,2} b_{k,2} b_{\ell,1} \big] \label{first_c}
    \end{align}
    On the other hand, 
    \begin{align}
        &\mathrel{\phantom{=}} \sum_{k < \ell} \det(A_{\{k,\ell\} \times \{1,2\} }) \det(B_{\{ k,\ell \} \times \{1,2\}}) \\ 
        &= \sum_{k < \ell} \bigl( a_{k,1} a_{\ell,2} - a_{k,2} a_{\ell,1} \bigr) \bigl( b_{k,1} b_{\ell,2} - b_{k,2} b_{\ell,1} \bigr) \\
        &= \sum_{k < \ell} \bigl[a_{k,1} a_{\ell,2} b_{k,1} b_{\ell,2} - a_{k,1} a_{\ell,2} b_{k,2} b_{\ell,1} \bigr] + \underbrace{\sum_{k < \ell} \bigl[ a_{k,2}a_{\ell,1} b_{k,2} b_{\ell,1} - a_{k,2} a_{\ell,1} b_{k,1} b_{\ell,2}\bigr]}_{= \sum_{\ell < k} [ a_{\ell,2}a_{k,1} b_{\ell,2} b_{k,1} -  a_{\ell,2} a_{k,1} b_{\ell,1} b_{k,2}]} \\
        &= \sum_{k \neq \ell} \big[ a_{k,1}a_{\ell,2}b_{k,1}b_{\ell,2} - a_{k,1} a_{\ell,2} b_{k,2}b_{\ell,1} \big]. \label{second_c}
    \end{align}
    \eqref{first_c} and \eqref{second_c} are equal, and the result follows.
\end{proof}

\begin{proposition}\thlabel{inner_product_existence}
    There exists an inner product on $\Lambda^2\R^4$ such that for all $P,Q \in G_2^+(\R^4)$ we have $\langle \omega_P, \omega_Q\rangle$ equals the determinant of the orthogonal projection map $P\to Q$.
\end{proposition}
\begin{proof}
    Let $\{e_1, e_2, e_3, e_4\} \subseteq \R^4$ be the standard basis, and let $\{ e_1^*, e_2^*, e_3^*, e_4^* \}\subseteq \Lambda^1\R^4$ be the corresponding dual basis. Consider the inner product on $\Lambda^2\R^4$ defined by declaring the basis 
    \begin{equation}
        \{ e_1^* \wedge e_2^*, \; e_1^* \wedge e_3^*, \; e_1^* \wedge e_4^*, \; e_2^* \wedge e_3^*, \; e_2^* \wedge e_4^*, \; e_3^* \wedge e_4^* \} \subseteq \Lambda^2\R^4
    \end{equation}
    given by Proposition~\ref{basis_of_exterior} to be orthonormal. Let $P, Q \subseteq \R^4$ be a pair of oriented 2-planes, and let $\pi: P \to Q$ be the orthogonal projection map. Our goal is to show that $\langle \omega_P, \omega_Q \rangle = \det(\pi)$.
    
    Let $p_1, p_2$ be a positively oriented orthonormal basis for $P$, and let $A$ be the $4 \times 2$ matrix whose columns are $p_1$ and $p_2$. Similarly, let $q_1, q_2$ be a positively oriented orthonormal basis for $Q$, and let $B$ be the $4 \times 2$ matrix whose columns are $q_1$ and $q_2$. In other words, 
    \begin{equation}
        A = \begin{bmatrix}
            e_1^*(p_1) & e_1^*(p_2) \\
            e_2^*(p_1) & e_2^*(p_2) \\
            e_3^*(p_1) & e_3^*(p_2) \\
            e_4^*(p_1) & e_4^*(p_2)
        \end{bmatrix}, \qquad B = \begin{bmatrix}
            e_1^*(q_1) & e_1^*(q_2) \\
            e_2^*(q_1) & e_2^*(q_2) \\
            e_3^*(q_1) & e_3^*(q_2) \\
            e_4^*(q_1) & e_4^*(q_2)
        \end{bmatrix}.
    \end{equation}
    Since $\pi(p_1) = \langle p_1, q_1\rangle q_1 + \langle p_1, q_2 \rangle q_2$ and $\pi(p_2) = \langle p_2, q_1\rangle q_1 + \langle p_2, q_2 \rangle q_2$, we find that $\pi$ has matrix representation
    \begin{equation}
        \begin{bmatrix}
            \langle p_1, q_1 \rangle & \langle p_2, q_1\rangle \\
            \langle p_1, q_2 \rangle & \langle p_2, q_2 \rangle
        \end{bmatrix},
    \end{equation}
    which equals $B^TA$. So
    \begin{equation}\label{bta}
        \det(\pi) = \det(B^TA). 
    \end{equation}
    On the other hand, we have 
    \begin{align}
        \langle \omega_P, \omega_Q\rangle &= \langle p_1^* \wedge p_2^*, \; q_1^* \wedge q_2^* \rangle \\
        &= \sum_{k < \ell} \Bigl( [p_1^* \wedge p_2^*](e_k, e_\ell) \Bigr) \Bigl( [q_1^* \wedge q_2^*](e_k, e_\ell) \Bigr) \\
        &= \sum_{k < \ell}\Bigl( p_1^*(e_k) p_2^*(e_\ell) - p_1^*(e_\ell) p_2^*(e_k) \Bigr)\Bigl( q_1^*(e_k) q_2^*(e_\ell) - q_1^*(e_\ell) q_2^*(e_k) \Bigr) \\
        &= \sum_{k < \ell} \det(A_{\{k,\ell\} \times \{ 1,2 \}}) \det(B_{\{k, \ell\} \times \{1, 2\}}). \label{second}
    \end{align}
    Applying Lemma~\ref{cauchy_binet} to equations \eqref{bta} and \eqref{second}, we get $\langle \omega_P, \omega_Q \rangle = \det(\pi)$, as desired.
\end{proof}
\begin{proposition}
    The inner product given by Proposition \textnormal{\ref{inner_product_existence}} is unique.
\end{proposition}
\begin{proof}
    Recall that to specify an inner product on a vector space, it's enough to specify it on a basis. We know how the inner product behaves on vectors in $\im(\omega)$, so to prove uniqueness, it suffices to show that $\im(\omega)$ contains a basis for $\Lambda^2\R^4$. This is true by Proposition~\ref{basis_of_exterior}.
\end{proof}

\subsection{Existence, uniqueness of the Hodge star operator}\label{hodge_star_existence_uniqueness}

Here we prove that there is exactly one linear map $*: \Lambda^2\R^4 \to \Lambda^2 \R^4$ which satisfies the property given in \thref{hodge_star}. Though we will only prove it for maps $\Lambda^2 \R^4 \to \Lambda^2 \R^4$, this easily generalizes to maps $\Lambda^k\R^n \to \Lambda^{n-k} \R^n$.

First, we set up some notation:
\begin{itemize}
    \item Let $[4]$ denote the set $\{1,2,3,4\}$.
    \item Let $S_4$ denote the set of bijections $[4] \to [4]$. 
    \item Given $\sigma \in S_4$, let $\sgn(\sigma)$ denote the sign of the permutation $\sigma$.
    \item Let $\binom{[4]}{2}$ denote the set of subsets of $[4]$ of cardinality 2. 
    \item Given $I \subseteq [4]$, let $\sgn_{[4]}(I)$ denote the sign of the permutation given by listing the elements of $I$ in increasing order, followed by the elements of $[4] \setminus I$ in increasing order.
    \item Given vectors $v_1, v_2, v_3, \ldots$ and a set of natural numbers $I = \{ i_1, \cdots, i_k \}$ with $i_1 < \cdots < i_k$, let $v_I$ denote the tuple $(v_{i_1}, \cdots, v_{i_k})$.
\end{itemize}
Recall that the wedge product $\wedge: \Lambda^2\R^4 \times \Lambda^2\R^4 \to \Lambda^4\R^4$ is defined by 
\begin{align}
    [\alpha \wedge \beta](v_1, v_2, v_3, v_4) &= \frac{1}{2!2!} \sum_{\sigma \in S_4} \sgn(\sigma) \, \alpha(v_{\sigma(1)}, v_{\sigma(2)}) \beta(v_{\sigma(3)}, v_{\sigma(4)}) \\
    &= \sum_{I \in \binom{[4]}{2}} \sgn_{[4]}(I) \, \alpha(v_I) \beta(v_{[4] \setminus I}).
\end{align}
\begin{lemma}\thlabel{volume_form}
    The multilinear map $(\R^4)^4\to \Lambda^4\R^4$ given by 
    \begin{equation}
        (p_1, p_2, p_3, p_4) \mapsto (p_1^* \wedge p_2^*) \wedge(p_3^* \wedge p_4^*)
    \end{equation}
    is alternating. Furthermore, if $(p_1, p_2, p_3, p_4)$ is a positively oriented orthonormal basis for $\R^4$, then 
    \begin{equation}
        (p_1^* \wedge p_2^*) \wedge (p_3^* \wedge p_4^*) = \operatorname{Vol},
    \end{equation}
    where $\operatorname{Vol} \in \Lambda^4\R^4$ is the function which maps a 4-tuple of vectors in $\R^4$ to the signed volume of the parallelepiped spanned by those vectors. 
\end{lemma}
\begin{proof}
    By definition, 
    \begin{align}
        &\mathrel{\phantom{=}} [(p_1^* \wedge p_2^*) \wedge (p_3^* \wedge p_4^*)](v_1, v_2, v_3, v_4) \\ 
        &= \sum_{I \in \binom{[4]}{2}} \sgn_{[4]}(I) \, [p_1^* \wedge p_2^*](v_I) [p_3^* \wedge p_4^*](v_{[4] \setminus I}) \\
        &= \sum_{I \in \binom{[4]}{2}} \sgn_{[4]}(I) \, \Bigl( \sum_{J \in \binom{I}{1}} \sgn_{I}(J) \, p_1^*(v_J)p_2^*(v_{I \setminus J}) \Bigr) \Bigl( \sum_{K \in \binom{[4] \setminus I}{1}} \sgn_{[4] \setminus I}(K) \, p_3^*(v_K)p_4^*(v_{[4] \setminus I \setminus K}) \Bigr) \\
        &= \sum_{I \in \binom{[4]}{2}} \sum_{J \in \binom{I}{1}} \sum_{K \in \binom{[4] \setminus I}{1}} \sgn_{[4]}(I) \sgn_{I}(J) \sgn_{[4] \setminus I}(K) \, p_1^*(v_J)p_2^*(v_{I \setminus J})p_3^*(v_K)p_4^*(v_{[4] \setminus I \setminus K}).
    \end{align}
    We observe that $\sgn_{[4]}(I) \sgn_{I}(J) \sgn_{[4] \setminus I}(K) = \sgn(\sigma)$, where $\sigma \in S_4$ is the permutation 
    \begin{align}
        1 &\mapsto J, \\
        2 &\mapsto I \setminus J, \\
        3 &\mapsto K, \\
        4 &\mapsto [4] \setminus I \setminus K,
    \end{align}
    so the above expression reduces to
    \begin{gather}
        \sum_{\sigma \in S_4} \sgn(\sigma) \, p_1^*(v_{\sigma(1)})p_2^*(v_{\sigma(2)})p_3^*(v_{\sigma(3)})p_4^*(v_{\sigma(4)}) \\
        = \det\begin{bmatrix}
            p_1^*(v_1) & p_1^*(v_2) & p_1^*(v_3) & p_1^*(v_4) \\
            p_2^*(v_1) & p_2^*(v_2) & p_2^*(v_3) & p_2^*(v_4) \\
            p_3^*(v_1) & p_3^*(v_2) & p_3^*(v_3) & p_3^*(v_4) \\
            p_4^*(v_1) & p_4^*(v_2) & p_4^*(v_3) & p_4^*(v_4)
            \end{bmatrix}.
    \end{gather}
    Swapping any two $p_i$ would cause the above determinant to switch sign, and so our map is alternating. And if $(p_1, p_2, p_3, p_4)$ is a positively oriented orthonormal basis for $\R^4$, the above determinant equals $\operatorname{Vol}(v_1, v_2, v_3, v_4)$.
\end{proof}

\begin{lemma}\thlabel{unique_star_char}
    Fix $\beta \in \Lambda^2\R^4$. Then there is at most one choice of $\gamma \in \Lambda^2\R^4$ satisfying 
    \begin{equation}
        \alpha \wedge \gamma = \langle \alpha, \beta \rangle \operatorname{Vol} \qquad \text{for all } \alpha \in \Lambda^2\R^4.
    \end{equation}
\end{lemma}
\begin{proof}
    Let $\alpha = e_i^* \wedge e_j^*$ with $i < j$. Writing $\gamma = \sum_{k < \ell} c_{k,\ell}[e_k^* \wedge e_\ell^*]$, we have by Lemma~\ref{volume_form} that wedging with $\alpha$ kills all but one of the terms in the sum. It follows that each $c_{k, \ell}$ is uniquely determined, and so $\gamma$ is uniquely determined.
\end{proof}

\begin{proposition}\thlabel{hodge_star_existence}
    There exists a linear map $*: \Lambda^2 \R^4 \to \Lambda^2\R^4$ which satisfies $*\omega_P = \omega_{P^\perp}$, where $P^\perp$ is the orthogonal complement of $P$ oriented so that if $(p_1, p_2)$ is a positive basis for $P$ and $(p_3, p_4)$ is a positive basis for $P^\perp$, then $(p_1, p_2, p_3, p_4)$ is a positive basis for $\R^4$.
\end{proposition}
\begin{proof}
    Proposition~\ref{basis_of_exterior} tells us that $\{e_i^* \wedge e_j^* \mid 1 \leq i < j \leq 4 \}$ is a basis for $\Lambda^2\R^4$, so we may define $*$ to be the linear map sending each basis element $e_i^* \wedge e_j^*$ with $i < j$ to $\sgn_{[4]}(\{i,j\})[e_k^* \wedge e_\ell^*]$, where $\{i,j,k,\ell \} = [4]$ and $k < \ell$. 
    
    Fix $\alpha, \beta \in \{e_i^* \wedge e_j^* \mid 1 \leq i < j \leq 4 \}$. By Lemma~\ref{volume_form} we have
    \begin{equation}
        \alpha \wedge *\beta = \begin{cases}
            \operatorname{Vol} & \text{for $\alpha = \beta$,} \\
            0 & \text{for $\alpha \neq \beta$,}
        \end{cases}
    \end{equation}
    and by \thref{inner_product} we have 
    \begin{equation}
        \langle \alpha, \beta \rangle = \begin{cases}
            1 & \text{for $\alpha = \beta$,} \\
            0 & \text{for $\alpha \neq \beta$.}
        \end{cases}
    \end{equation}
    Therefore, \begin{equation}\label{star_characterization}
        \alpha \wedge *\beta = \langle \alpha, \beta \rangle \operatorname{Vol} \qquad \text{for all } \alpha, \beta \in \{e_i^* \wedge e_j^* \mid 1 \leq i < j \leq 4 \}.
    \end{equation}
    By linearity, we find that this equality holds for all $\alpha, \beta \in \Lambda^2\R^4$. In particular, given $P \in G_2^+(\R^4)$ we have
    \begin{equation}\label{special_star_characterization}
        \alpha \wedge *\omega_P = \langle \alpha, \omega_P \rangle \operatorname{Vol} \qquad \text{for all } \alpha \in \Lambda^2\R^4.
    \end{equation}

    Now fix a 2-plane $P \in G_2^+(\R^4)$, let $(p_1, p_2)$ be a positively oriented orthonormal basis for $P$, and let $(p_3, p_4)$ be a positively oriented orthonormal basis for $P^\perp$. 

    Fix $\alpha \in \{ p_i^* \wedge p_j^* \mid 1 \leq i < j \leq 4 \}$. By Lemma~\ref{volume_form} we have
    \begin{equation}
        \alpha \wedge \omega_{P^\perp} = \begin{cases}
            \operatorname{Vol} & \text{for $\alpha = \omega_P$,} \\
            0 & \text{for $\alpha \neq \omega_P$,}
        \end{cases}
    \end{equation}
    and by \thref{inner_product} we have 
    \begin{equation}
        \langle \alpha, \omega_P \rangle = \begin{cases}
            1 & \text{for $\alpha = \omega_P$,} \\
            0 & \text{for $\alpha \neq \omega_P$.}
        \end{cases}
    \end{equation}
    Therefore, \begin{equation}\label{perp_characterization}
        \alpha \wedge \omega_{P^\perp} = \langle \alpha, \omega_P \rangle \operatorname{Vol} \qquad \text{for all } \alpha \in \{ p_i^* \wedge p_j^* \mid 1 \leq i < j \leq 4 \}.
    \end{equation}
    By linearity, we find that this equality holds for all $\alpha \in \Lambda^2\R^4$: \begin{equation}\label{special_perp_characterization}
        \alpha \wedge \omega_{P^\perp} = \langle \alpha, \omega_P \rangle \operatorname{Vol} \qquad \text{for all } \alpha \in \Lambda^2\R^4.
    \end{equation}
    Applying Lemma~\ref{unique_star_char} to \eqref{special_star_characterization} and \eqref{special_perp_characterization}, we find that $*\omega_P = \omega_{P^\perp}$ for all $P \in G_2^+(\R^4)$, and so our map does satisfy the desired property.
\end{proof}

\begin{proposition}
    The map given by Proposition \textnormal{\ref{hodge_star_existence}} is unique.
\end{proposition}
\begin{proof}
    Recall that to specify a linear map, it's enough to specify where the map sends elements of a basis. We know where $*$ sends vectors in $\im(\omega)$, so to prove uniqueness, it suffices to show that $\im(\omega)$ contains a basis for $\Lambda^2\R^4$. This is true by Proposition~\ref{basis_of_exterior}.
\end{proof}

\subsection{Darboux normal form}
As we observed in \thref{omega}, the map $\R^4 \times \R^4\to \R$ which takes a pair of vectors to the signed area of the parallelogram spanned by their projection onto some oriented 2-plane is an example of an alternating bilinear map in $\Lambda^2\R^4$. However, not all maps in $\Lambda^2\R^4$ arise in this way. Our goal in this section is to prove Proposition~\ref{darboux}, which tells us that to express any alternating bilinear map in $\Lambda^2\R^4$, we only need to take a linear combination of at most two of these projection maps. 

\begin{lemma}\thlabel{gram_schmidt}
    Let $\alpha, \beta \in \Lambda^1\R^4$, and let $\varphi: \R^4 \to \Lambda^1\R^4$ be the canonical map to the dual which sends each vector $x$ to the map $[w \mapsto \langle w, x\rangle]$. Then $\alpha \wedge \beta = a\omega_{\varphi^{-1}(\lspan(\alpha,\beta))}$ for some $a \in \R$.
\end{lemma}
\begin{proof}
    We equip $\Lambda^1\R^4$ with its natural inner product by declaring the map $\varphi$ to be orthogonal. We notice from the definition of the wedge product that the following two identities hold for all $c \in \R$:
    \begin{enumerate}
        \item $\alpha \wedge (\beta + c\alpha) = \alpha \wedge \beta$, 
        \item $\alpha \wedge c\beta = c(\alpha \wedge \beta)$.
    \end{enumerate}
    Therefore, we can perform Gram-Schmidt orthogonalization on the pair $\alpha, \beta \in \Lambda^1\R^4$ to get an orthonormal pair $\alpha', \beta' \in \Lambda^1\R^4$ for which $\alpha \wedge \beta$ is some real-number multiple of $\alpha' \wedge \beta'$. By \thref{omega} we have 
    \begin{align}
        \alpha' \wedge \beta' &= \omega_{\lspan(\varphi^{-1}(\alpha'), \varphi^{-1}(\beta'))} \\ &= \omega_{\varphi^{-1}(\lspan(\alpha', \beta'))} \\ &= \omega_{\varphi^{-1}(\lspan(\alpha, \beta))}.
    \end{align}
    This completes the proof.
\end{proof}

\begin{proposition}[Darboux normal form]\thlabel{darboux}
    Every alternating bilinear map in $\Lambda^2\R^4$ may be written in the form $a\omega_P + b \omega_Q$, where $P, Q \in G_2^+(\R^4)$ are planes which intersect trivially and $a,b \in \R$ are constants.
\end{proposition}
\begin{proof}
    Let $\mathcal{E} = \{e_1, e_2, e_3, e_4\} \subseteq \R^4$ be the standard basis, and let $\alpha: \R^4 \times \R^4 \to \R$ be an alternating bilinear map. Given a vector 
    \begin{equation}
        u = \begin{bmatrix}
            u^1 \\ u^2 \\ u^3 \\ u^4
        \end{bmatrix} \in \R^4,
    \end{equation}
    we write $u_\mathcal{E}$ as shorthand for $\sum_{i = 1}^4 u^i e_i$.
    
    We may characterize $\alpha$ in the basis $\mathcal{E}$ as follows: Let $A$ be the $4 \times 4$ matrix whose $(i,j)^\text{th}$ entry is $\alpha(e_i, e_j)$. Then for all $u,v \in \R^4$, we have 
    \begin{equation}
        \alpha(u_\mathcal{E}, v_\mathcal{E}) = u^T A v.
    \end{equation}
    Since $\alpha$ is alternating, $A$ must be skew-symmetric. Now suppose we had another basis $\mathcal{F} = \{ f_1, f_2, f_3, f_4 \}$ with $Q$ being the change of basis matrix from $\mathcal{E}$ to $\mathcal{F}$ (that is, $f_j = \sum_i Q_{ij} e_i$). We may characterize $\alpha$ in the basis $\mathcal{F}$ as follows:
    \begin{align}
        \alpha(u_\mathcal{F}, v_\mathcal{F}) &= (Qu)^T A (Qv) \\
        &= u^T(Q^TAQ) v.
    \end{align}
    If we view right-multiplication by $Q$ as performing a sequence of elementary column operations on $A$, then left-multiplication by $Q^T$ has the effect of performing the corresponding sequence of elementary row operations. We call the action of performing an elementary column operation and then the corresponding elementary row operation an \emph{elementary congruence operation}. It's not hard to see that any $4 \times 4$ skew-symmetric matrix may be taken to one of the following three forms via a finite sequence of elementary congruence operations:
    \begin{equation}
        B_0 = \begin{bmatrix}
            0 & 0 & 0 & 0 \\ 
            0 & 0 & 0 & 0 \\
            0 & 0 & 0 & 0 \\ 
            0 & 0 & 0 & 0
        \end{bmatrix}, \quad
        B_1 = \begin{bmatrix}
            0 & 1 & 0 & 0 \\ 
            -1 & 0 & 0 & 0 \\
            0 & 0 & 0 & 0 \\ 
            0 & 0 & 0 & 0
        \end{bmatrix}, \quad
        B_2 = \begin{bmatrix}
            0 & 1 & 0 & 0 \\ 
            -1 & 0 & 0 & 0 \\
            0 & 0 & 0 & 1 \\ 
            0 & 0 & -1 & 0
        \end{bmatrix}.
    \end{equation}
    Therefore, there exists a choice of invertible matrix $Q$ for which $Q^T A Q = B_i$ for some $i \in \{0,1,2\}$, which implies there exists a choice of $\mathcal{F}$ for which 
    \begin{equation}
        \alpha(u_\mathcal{F}, v_\mathcal{F}) = u^TB_i v \qquad \text{for all $u,v \in \R^4$.}
    \end{equation}
    We now split into cases:
    \begin{itemize}
        \item If $i = 0$ then $\alpha = 0$. So we can set $a = b = 0$, and we're done.
        \item If $i = 1$ then we define $\varphi_\mathcal{E},\varphi_\mathcal{F}: \R^4 \to \Lambda^1\R^4$ to be the canonical maps to the dual associated with the bases $\mathcal{E}$ and $\mathcal{F}$, respectively. In other words, $\varphi_\mathcal{E}(e_i)(e_j) = \varphi_{\mathcal{F}}(f_i)(f_j) = \delta_{ij}$.

        The equation $\alpha(u_\mathcal{F}, v_\mathcal{F}) = u^T B_1 v$ tells us that
        \begin{equation}
            \alpha(f_i, f_j) = \begin{cases}
                1 & \text{for $(i,j) = (1,2)$,} \\
                -1 & \text{for $(i,j) = (2,1)$,} \\
                0 & \text{otherwise.}
            \end{cases}
        \end{equation}
        Therefore, it's clear that $\alpha = \varphi_\mathcal{F}(f_1) \wedge \varphi_\mathcal{F}(f_2)$. By Lemma~\ref{gram_schmidt} we have $\alpha = a \omega_P$, where $a$ is a real number and
        \begin{align}
            P &= \varphi_{\mathcal{E}}^{-1}(\lspan(\varphi_\mathcal{F}(f_1) , \varphi_\mathcal{F}(f_2))) \\
            &= \varphi_\mathcal{E}^{-1} \circ \varphi_\mathcal{F}(\lspan(f_1, f_2)).
        \end{align}
        Setting $b = 0$, we are done.
        \item If $i = 2$ then we find that $\alpha = \varphi_\mathcal{F}(f_1) \wedge \varphi_\mathcal{F}(f_2) + \varphi_\mathcal{F}(f_3) \wedge \varphi_\mathcal{F}(f_4)$. By the same argument as in the $i = 1$ case, we may rewrite this as $\alpha = a \omega_P + b \omega_Q$, where $P = \varphi_\mathcal{E}^{-1} \circ \varphi_\mathcal{F}(\lspan(f_1, f_2))$ and $Q = \varphi_\mathcal{E}^{-1} \circ \varphi_\mathcal{F}(\lspan(f_3, f_4))$. Since $\lspan(f_1, f_2)$ and $\lspan(f_3, f_4)$ intersect trivially, we conclude that $P$ and $Q$ intersect trivially, and we are done.
    \end{itemize}
\end{proof}

\end{document}